\documentclass[a4paper,10pt]{article}
\usepackage{bbm}

\usepackage{dsfont}
\usepackage[breaklinks=true]{hyperref}
\usepackage{amsfonts}
\usepackage{amsmath,amsthm,amscd,amssymb,graphicx}\usepackage{mathrsfs}
\usepackage[sort&compress, square, numbers]{natbib}
\newtheorem{thm}{Theorem}[section]

\newtheorem{lem}[thm]{Lemma}
\newtheorem{prop}[thm]{Proposition}
\theoremstyle{definition}

\theoremstyle{remark}
\newtheorem{rem}[thm]{Remark}
\numberwithin{equation}{section}

\newcommand{\abs}[1]{\left\lvert#1\right\rvert}

\newcommand{\ind}[1]{\mathbf{1}_{\{#1\}}}
\renewcommand{\P}{\mathbb{P}}
\newcommand{\E}{\mathbb{E}}
\newcommand{\R}{\mathbb{R}}
\newcommand{\T}{\mathbb{T}}
\newcommand{\N}{\mathbb{N}}
\newcommand{\J}{\mathbb{J}}
\newcommand{\K}{\mathbb{K}}
\newcommand{\I}{\mathscr{I}}
\newcommand{\F}{\mathscr{F}}
\newcommand{\D}{\mathscr{D}}
\newcommand{\KM}{K_{\mathrm{max}}}

 \allowdisplaybreaks

\begin{document}
\title{
    {Exact convergence rates in central limit theorems for  a branching random walk   with a   random  environment in time\thanks{The project is partially supported
by the National Natural Science Foundation of China (NSFC, Grants No. 11101039, No. 11171044, No. 11271045, and No. 11401590), by a cooperation program
between NSFC and CNRS of France (Grant No. 11311130103), by the Fundamental Research Funds for the Central Universities (2013YB52) and by Hunan Provincial Natural Science Foundation of China
(Grant No. 11JJ2001).} }%
      } 

\author{Zhiqiang Gao\thanks
 {School of Mathematical Sciences, Laboratory of Mathematics and Complex Systems, Beijing Normal University, Beijing 100875, P. R. China (gaozq@bnu.edu.cn)},   Quansheng Liu \thanks{Corresponding author,  Univ. Bretagne-Sud, CNRS UMR 6205, LMBA, campus de Tohannic, F-{56000} Vannes, France  and
 Changsha University of Science and Technology, School of Mathematics and Computing Science, Changsha {410004}, China(quansheng.liu@univ-ubs.fr)}}%
\date{\today}%

\maketitle

\begin{abstract}
Chen [Ann. Appl. Probab. {\bf 11} (2001), 1242--1262] derived  exact  convergence rates in  a central limit theorem and a local limit theorem for a supercritical branching Wiener process.
We extend Chen's results to a  branching random walk  under weaker moment conditions. For the branching Wiener  process, our results sharpen Chen's
 by relaxing the second moment condition used by Chen to a moment condition of the form $ \E X (\ln^+X )^{1+\lambda}< \infty$.
 In the rate functions that we find for a  branching random walk, we figure out some new terms which didn't appear in Chen's work.
 The results are established in the more general framework, i.e. for a  branching random walk with a random environment in time.
 The lack of the second moment condition for the offspring distribution and the fact that the exponential moment does not exist necessarily for the displacements make the proof  delicate; the difficulty is overcome by a careful analysis of martingale convergence using a truncating argument. The analysis is significantly more awkward due to the  appearance of the random environment.



\end{abstract}

\section{Introduction}

The theory of branching random walk has been studied by
many authors. It plays
an important role, and is closely related to many problems arising
in a variety of applied probability setting, including
 branching processes, multiplicative
 cascades, infinite particle systems,  Quicksort algorithms
  and random fractals (see e.g. \cite{Liu98, Liu00}).   For recent developments of the subject, see e.g. Hu and Shi \cite{HS09}, Shi \cite{Shi2012}, Hu \cite{Hu14}, Attia and Barral \cite{Barral14} and the references therein.

In the classical branching random walk, the point processes indexed
by the particles $u$, formulated by the number of its children and
their displacements, have a fixed constant distribution  for all
particles $u$. In reality this distributions may vary from
generation to generation according to a random environment, just as
in the case of a branching process in random environment introduced
in \cite{SmithWilkinson69, AthreyaKarlin71BPRE1, AthreyaKarlin71BPRE2}. In other words, the
distributions themselves may be realizations of a stochastic
process, rather than being fixed.  This property makes the model be closer to the reality compared to the classical branching random walk.
In this paper, we shall consider such a model, called \emph{a branching random walk with a random environment in time} .


Different kinds of branching random walks in  random environments have been introduced and studied in the literature.
 Baillon,  Cl{\'e}ment, Greven and den Hollander \cite{BaillonClementGrevenHollander93,GrevenHollander92PTRF}  considered the case  where the
offspring distribution of a particle situated at $z\in \mathbb{Z}^d$
depends on a random environment indexed by the location $z$, while the moving
mechanism is controlled by a fixed deterministic law.  Comets and Popov  \cite{CometsPopov2007AOP, CometsPopov2007ALEA}  studied  the case where
 both the offspring distributions and the moving laws depend on a random environment indexed by the location.  In the model studied in \cite{BirknerGK05,HuYoshida09,Nakashima11,Yoshida08,CometsYoshida2011JTP},  the offspring distribution of a particle of generation $n$ situated at $z\in \mathbb{Z}^d (d\geq 1)$
depends on a random space-time environment indexed by $\{(z, n)\}$, while each particle performs a  simple symmetric random walk on $d$-dimensional integer lattice $\mathbb{Z}^d (d\geq 1)$.  The model that we study in this paper  is different from those mentioned above.
It should also be mentioned that recently another  different  kind of branching random walks in time inhomogeneous environments has been considered extensively, see e.g.  Fang and Zeitouni (2012, \cite{FangZeitouni2012}), Zeitouni (2012, \cite{Zeitouni2012}) and Bovier and Hartung(2014, \cite{Bovier14}).
The readers may refer to these articles and  references therein for more information.

Denote by $Z_n(\cdot)$ the counting measure which counts the number of particles of generation $n$ situated in a given set.
For the classical branching random walk, a central limit theorem on $Z_n(\cdot)$, first conjectured by Harris (1963, \cite{Harris63BP}),  was shown by Asmussen and Kaplan (1976, \cite{AsmussenKaplan76BRW1,AsmussenKaplan76BRW2}),  and then extended to a general case by Klebaner (1982, \cite{Klebaner82AAP}) and
 Biggins (1990, \cite{Biggins90SPA});  for a  branching Wiener process,  R\'ev\'esz (1994,\cite{Revesz94}) studied  the convergence rates in the central limit theorems and  conjectured the exact convergence rates, which were confirmed by Chen (2001,\cite{Chen2001}).    Kabluchko (2012,\cite{Kabluchko12}) gave an alternative proof of Chen's results under slightly stronger hypothesis.  R\'ev\'esz, Rosen and Shi (2005,\cite{ReveszRosenShi2005}) obtained a large time asymptotic expansion in the local limit theorem  for branching Wiener processes,  generalizing Chen's result.

The  first objective of our present paper is to  extend Chen's results to the branching random walk under weaker moment conditions. In our results  about the exact convergence rate in the central limit theorem and the local limit theorem, the rate functions that we find  include some new terms  which didn't appear in Chen's paper \cite{Chen2001}.
 In Chen's work, the second moment condition was assumed for the offspring distribution. Although the setting we  consider now  is much more general, in our results the second moment  condition will be relaxed to a moment condition  of the form
$\E X (\ln^+ X) ^{1+\lambda} < \infty$ . It has been well known that in branching random walks, such a relaxation is quite delicate. Another interesting aspect is that we do not assume the existence of exponential moments for the moving law, which holds automatically in the case of the branching Wiener process. The lack of the second moment condition (resp. the exponential moment condition) for the offspring distribution (resp. the moving law) makes the proof delicate. The difficulty will be overcome via a careful analysis of the convergence of some associated martingales using truncating arguments.

 The second objective of our paper is to extend the results to the branching random walk with a random environment in time.
This  model first appeared  in Biggins and Kyprianou (2004, \cite[Section 6]{BigginsKyprianou04AAP}),  where a criterion was given for the non-degeneration of the limit of the natural martingale; see also  Kuhlbusch (2004, \cite{Ku04})  for the equivalent form of the criterion on weighted branching processes in random environment.
For $Z_n(\cdot)$ and related quantities on this model, Liu (2007,\cite{Liu07ICCM}) surveyed several limit theorems, including large deviations theorems and a law of large numbers  on the rightmost particle.  In   \cite{GLW14}, Gao, Liu and Wang   showed a central limit theorem on the counting measure $Z_n(\cdot)$ with appropriate norming. Here we  study the convergence rate in the central limit theorem and a local limit theorem for  $Z_n(\cdot)$.   Compared with the classical branching random walk,
 the approach is significantly more difficult due to  the appearance of the random environment.
%
%



The article is organized as follows. In Section \ref{sec2},  we give a rigorous description of the model and   introduce the basic assumptions and notation, then we formulate  our main results  as Theorems \ref{th1} and \ref{th2}. In Section \ref{sec6}, we introduce some notation and recall a theorem on the Edgeworth expansions for sums of independent random variables  used in our  proofs.  We give the proofs of the main theorems in Section \ref{sec3} and \ref{sec4}, respectively. Whilst Section \ref{sec5} will be  devoted to the proofs of the reminders.

\section{Description of the model and the main results}\label{sec2}

\subsection{Description of the model}
We describe the model as follows (\cite{Liu07ICCM,GLW14}).
 \emph{A random environment in time} $\xi=(\xi_n)$ is formulated as a sequence of random variables independent and identically distributed with values in some measurable space $(\Theta,\mathcal{F})$. Each realization of $ \xi_n $ corresponds to two probability distributions: the offspring distribution $p(\xi_n)
 = (p_0(\xi_n), p_1(\xi_n), \cdots ) $ on  $\N = \{0,1, \cdots\}$, and the moving distribution $ G (\xi_n) $ on $\R$.
 Without loss of generality, we can take $\xi_n $ as coordinate functions defined on the product space $(\Theta^{\mathbb{N}}, \mathcal{F}^{\otimes\mathbb{N}})$
equipped with the product law $\tau$  of some probability law $\tau_0$ on $(\Theta, \cal F)$,
 which is invariant and ergodic under the usual shift transformation
 $\theta$ on $\Theta^{\mathbb{N}}$: $\theta(\xi_0,\xi_1,\cdots)= (\xi_1,\xi_2,\cdots) $.

When the environment $\xi=(\xi_n)$ is given, the process can be described as follows.
It
begins at time $ 0$  with one initial particle $\varnothing$ of
generation $0 $ located at $S_{\varnothing} = 0 \in \mathbb{R}$; at time
$1$, it is replaced by $N = N_{\varnothing}$ new particles $ \varnothing i = i( 1\leq i\leq  N) $ of generation
$1$, located at $S_i = L_{\varnothing i} (1\leq i\leq  N),$ where
$N, L_1, L_2,  \cdots $ are mutually  independent, $N$ has the law $p(\xi_0)$,  and each $ L_i$ has the law $G(\xi_0)$.
 In general,
 each particle $u= u_1...u_n$ of generation $n$ is replaced at time $ n + 1 $ by $N_{u} $ new particles $ u i (1\leq i \leq N_u) $ of generation $n+1$,
  with displacements $L_{u i} (1\leq i \leq N_u) $, so that the $i$-th
child  $ u i  $ is located at  $$S_{u i}=S_{u}+ L_{u i},$$ where
$N_u,  L_{u1}, L_{u2}, \cdots  $  are mutually  independent, $N_u$ has the law $ p(\xi_n)$, and each $L_{ui}$  has the same law $G(\xi_n)$. By definition,
  given the environment $\xi$, the random variables $N_u$  and $L_u$, indexed by all the finite sequences $u$ of positive integers, are independent of each other.

For each realization $\xi \in \Theta^\N$ of the environment sequence,
let $(\Gamma, {\cal G},  \mathbb{P}_\xi)$ be the probability space under which the
process is defined (when the environment $\xi$ is fixed to the given realization).  The probability
$\mathbb{P}_\xi$ is usually called \emph{quenched law}.
The total probability space can be formulated as the product space
$( \Theta^{\mathbb{N}}\times\Gamma , {\cal E}^{\N} \otimes \cal G,   \mathbb{P})$,
 where $ \mathbb{P} = \E  (\delta_\xi \otimes \mathbb{P}_{\xi}) $ with $\delta_\xi $ the Dirac measure at $\xi$ and $\E$ the expectation with respect to the random variable $\xi$, so that  for all measurable and
 positive $g$ defined on $\Theta^{\mathbb{N}}\times\Gamma$, we have
  \[\int_{ \Theta^{\mathbb{N}}\times\Gamma } g (x,y) d\mathbb{P}(x,y) = \E  \int_\Gamma g(\xi,y) d\mathbb{P}_{\xi}(y).\]
The total
probability $\P$ is usually called \emph{annealed law}.
The quenched law $\P_\xi$ may be considered to be the conditional
probability of $\P$ given $\xi$. The expectation with respect to $\mathbb{P}$ will still be denoted by $\E$; there will be no confusion for reason of consistence.   The expectation with respect to
$\P_\xi$ will be denoted by $\E_\xi$.

Let $\mathbb{T}$ be the genealogical tree with $\{N_u\}$ as defining elements. By definition, we have:
(a) $\varnothing\in \mathbb{T}$; (b) $ui \in \mathbb{T}$ implies $u\in \mathbb{T}$; (c) if $ u\in \mathbb{T} $, then $ui\in \mathbb{T} $
if and only if $1\leq i\leq N_u $.
 Let $$ \mathbb{T}_n =\{u\in
\mathbb{T} :|u|=n\} $$
 be the set of particles of generation $n$, where $|u|$ denotes the length of the
sequence $u$ and represents the number of generation to which $u$ belongs.
\subsection{Main results}
Let  $Z_n(\cdot)$ be the counting measure of particles of generation $n$:
  for $B \subset \mathbb{R}$,
$$Z_n(B)= \sum_{u\in \mathbb{T}_n}  \mathbf{1}_{ B}(S_u).$$
 Then  $\{ Z_n(\mathbb{R})\}$ constitutes a branching process
in a random environment (see e.g. \cite{AthreyaKarlin71BPRE1,AthreyaKarlin71BPRE2,SmithWilkinson69}).
For $n\geq 0$, let  $\widehat{N}_n$ (resp.  $\widehat{L}_n$)  be a random variable with distribution $p(\xi_n)$ (resp. $G(\xi_n)$) under the law $\P_\xi$, and   define
\begin{equation*}
   m_n= m(\xi_n)= \E_\xi \widehat{N}_n,\quad  \Pi_n = m_0\cdots m_{n-1}, \quad   \Pi_0=1.
\end{equation*}
It is well known that   the normalized sequence  $$W_n=\frac{1}{\Pi_n} Z_n(\mathbb{R}), \quad n\geq 1$$
    constitutes a martingale with respect to the filtration
    $(\mathscr{F}_n)$ defined by
    $$ \mathscr{F}_0=\{\emptyset,\Omega \}, \quad   \mathscr{F}_n =\sigma ( \xi, N_u:|u| < n), \mbox{ for  }n\geq 1. $$
Throughout the paper, we  shall always assume the following conditions:
 \begin{equation}\label{cbrweq1}
    \E \ln m_0>0   \quad { \mathrm{and}}\quad  {\E}\left[\frac{1}{m_0}\widehat{N}_0 \Big(\ln ^+{\widehat{N}_0 }\Big)^{1+\lambda}   \right]<\infty   ,
    \end{equation}
where the value of $\lambda >0$ is to be specified in the hypothesis of the theorems. Under these conditions, the underlying  branching process  $\{ Z_n(\mathbb{R})\}$  is  \emph{supercritical},    $ Z_n(\mathbb{R}) \rightarrow \infty$   with positive probability,
  and the limit
    \begin{equation*}
    W=\lim_n W_n
\end{equation*}
verifies $\E W =1$ and $W>0$  almost surely (a.s.) on the explosion event $\{Z_\infty \rightarrow \infty  \}$ (cf. e.g. \cite{AthreyaKarlin71BPRE2,Tanny1988SPA}).

For $n\geq 0$, define
\begin{eqnarray*}
   && l_n   = \E_\xi \widehat{L}_n,~~
   \sigma_{n}^{(\nu)} =\E_\xi \big(\widehat{L}_n- l_n\big)^\nu, \mbox{ for } \nu\geq 2;\\
    &&
    \ell_n= \sum_{k=0}^{n-1} l_{k},  \quad s_n^{(\nu)} = \sum_{k=0}^{n-1} \sigma_{k}^{(\nu)} , \mbox{ for } \nu\geq 2, \quad
 s_n  = \big(s_n^{(2)}\big)^{\frac{1}{2} }.
\end{eqnarray*}

We  will need   the following conditions on the motion of   particles:
\begin{equation}\label{cbrweq2-3}
      \P \Big( \limsup_{|t|\rightarrow \infty }\big|\E_\xi e^{it\widehat{L}_0}\big| <1 \Big) >0   \quad \mbox{  and } \quad   \E\big(| \widehat{L}_0 |^{\eta } \big) <\infty  , \end{equation}
where the value of $\eta>1 $ is to  be specified  in the hypothesis of the theorems.  The first hypothesis  means that Cram\'er's condition about the  characteristic function of $\widehat{L}_0$  holds with positive probability.




Let $\{N_{1,n}\}$ and $\{N_{2,n}\}$ be two sequences of random variables, defined  respectively by
\begin{equation*}
N_{1,n} = \frac{1}{\Pi_n} \sum_{u\in \T_n} (S_u-\ell_n)\quad \mbox{and} \quad N_{2,n} = s_n^2 W_n- \frac{1}{\Pi_n} \sum_{u\in \T_n} (S_u-\ell_n)^2.
\end{equation*}
We shall prove that they are martingales with respect to
 the filtration $ (\mathscr{D}_n ) $  defined by $$ \mathscr{D}_0=\{\emptyset,\Omega \}, \quad  \mathscr{D}_n  =  \sigma ( \xi, N_u, L_{ui}:  i\geq 1, |u| <
  n),  \mbox{ for  $n\geq 1$}.$$

More precisely,we have the following  propositions.

\begin{prop}\label{th1a} Assume \eqref{cbrweq1} and $\E \big(\ln^- m_0\big)^{1+\lambda}<\infty$ for some $ \lambda>1$,  and
   $ \E\big(| \widehat{L}_0 |^{\eta } \big) <\infty  $ for some  $ \eta>2$.
Then    the sequence $\{(N_{1,n}, \mathscr{D}_n) \}$ is a martingale  and converges a.s.:
    $$V_1:= \displaystyle \lim_{n\rightarrow\infty}N_{1,n}  \mbox{ exists  a.s. in  } \R.
$$
\end{prop}

\begin{prop}
  \label{th2a}
  Assume \eqref{cbrweq1} and   $\E \big(\ln^- m_0\big)^{1+\lambda}<\infty$ for some $ \lambda>2$ ,
   and $ \E\big(| \widehat{L}_0 |^{\eta } \big) <\infty  $ for some  $ \eta>4$.
Then    the sequence $\{(N_{2,n}, \mathscr{D}_n) \}$ is a martingale  and converges a.s.:
    $$V_2 := \displaystyle \lim_{n\rightarrow\infty}N_{2,n}  \mbox{ exists  a.s. in  } \R.
$$
\end{prop}


Our main results are the following two theorems.   The first theorem concerns the exact convergence rate in the central limit theorem about the counting measure $Z_n$, while the second one is a local limit theorem. We shall use the notation
$$ Z_n(t)=Z_n((-\infty, t]),  \;\;  \phi(t)=\frac{1}{\sqrt{2\pi}}e^{-t^2/2}, \;\;  \Phi(t) = \int_{-\infty}^{t}\phi(x) \mathrm{d}x, \quad t\in \R. $$

\begin{thm}\label{th1}
Assume    \eqref{cbrweq1} for some $\lambda>8$,  \eqref{cbrweq2-3}  for some $ \eta>12$  and $\E m_0^{-\delta}<\infty$ for some $\delta>0.$    Then  for all $ t\in \R$,
 \begin{equation} \label{cbrweq4}
   \sqrt{n}\Big[\frac{1}{\Pi_n} Z_n(\ell_n+s_n t)  - \Phi(t) W\Big] \xrightarrow{n \rightarrow \infty } \mathcal{V}(t)  \quad \mbox {a.s.},
 \end{equation}
 where
 \begin{align*}
   \mathcal{V} (t) =- \frac{ \phi(t)\; V_1 }{  (\E \sigma_0^{(2)})^{1/2} }   +\frac{(\E \sigma_0^{(3)}) \,  (1-t^2)\; \phi(t) \; W }{6 (\E \sigma_0^{(2)})^{3/2} } .
 \end{align*}
\end{thm}

\begin{thm}\label{th2}
Assume   \eqref{cbrweq1}for some $\lambda>16$, \eqref{cbrweq2-3} for some $ \eta >16$  and $\E m_0^{-\delta}<\infty$ for some $\delta>0.$   Then for any bounded measurable set $A\subset \R$ with   Lebesgue measure  $|A|>0$,
\begin{equation}\label{cbrweq4a}
{n}\Bigg[ \sqrt{2\pi}s_n\Pi_n^{-1} Z_n(A+\ell_n)  - W\int_A e^{-\frac{x^2}{2s_n^2}} dx\Bigg] \xrightarrow{n \rightarrow \infty } \mu(A)   \quad \mbox {a.s.},
\end{equation}
where
$$ \mu(A) = \frac{|A|} { 2 \E \sigma_0^{(2)} } \Big( V_2 + 2 \; \overline{x}_A V_1\Big) + \frac{  |A|    \; c(A) }{8 (\E \sigma_0^{(2)})^{2}}
$$
 with $\displaystyle \overline{x}_A=  \frac{1} { |A| }  \int_A xdx $\;  and
$$ c(A)  =     W \; \E\big(\sigma_0^{(4)}-3\big(\sigma_0^{(2)}\big)^2   \big)
~+ 4 \; ({\E \sigma_0^{(3)} })  (V_1 -\overline{x}_AW   ) -  \frac{5 (\E\sigma_0^{(3)})^2 }{ 3 \; \E \sigma_0^{(2)} }W. $$
\end{thm}


\begin{rem}
For a branching Wiener process,
 Theorems \ref{th1} and \ref{th2} improve  Theorems 3.1 and 3.2 of Chen (2001,\cite{Chen2001}) by relaxing the second moment condition used by Chen to the
 moment condition of the form $ \E X (\ln^+ X)^{1+\lambda} < \infty$  (cf.  (\ref{cbrweq1})). For a branching random walk with a constant or random  environment,
 the second terms in $\mathcal{V}(\cdot)$ and $\mu(\cdot)$ are new: they did not appear
 in Chen's results \cite{Chen2001} for a branching Wiener process; the reason is that in the case of a Brownian motion, we have $\sigma_0^{(3)}= \sigma_0^{(4)}-3\big(\sigma_0^{(2)}\big)^2=0$.

\end{rem}
\begin{rem} As will be seen in the proof,
  if we assume   an exponential moment condition for the motion law, then the  moment condition on the underlying branching mechanism  can be weakened: in that case, we only need to assume that  $\lambda>3/2$ in Theorem \ref{th1} and $\lambda >4$ in Theorem \ref{th2}.  In particular, for a branching Wiener process, Theorem \ref{th1} (resp. Theorem \ref{th2} ) is valid when  (\ref{cbrweq1}) holds for some $\lambda>3/2$  (resp. $\lambda >4$).

\end{rem}

\begin{rem} \label{rem-lattice}
When the  Cram\'er condition $ \P \Big( \limsup_{|t|\rightarrow \infty }\big|\E_\xi e^{it\widehat{L}_0}\big| <1 \Big) >0 $ fails,
the situation is different. Actually, while revising our manuscript we find that   a lattice version (about a branching random walk on $\mathbb{Z}$ in a constant environment,  for which the preceding condition fails) of Theorems \ref{th1} and \ref{th2}  has been established very recently in \cite{GK2015}.

\end{rem}

For simplicity and without loss of generality, hereafter we always assume that $l_n=0$ (otherwise, we only need to replace  $L_{ui}$ by $L_{ui}-l_n$) and hence $\ell_n=0$. In the  following, we will write $K_\xi $ for a constant depending on the environment,  whose value may vary from lines to lines.

\section{Notation and Preliminary results}\label{sec6}
In this section, we  introduce some notation  and important lemmas which will be used in the sequel.
\subsection{Notation}\label{sec3.1}

In addition to the $\sigma-$fields $\mathscr{F}_n$ and  $\mathscr{D}_n$,  the following  $\sigma$-fields  will also be used:
\begin{eqnarray*}
  \mathscr{I}_0=\{\emptyset,\Omega \}, \quad  \mathscr{I}_n &=&   \sigma ( \xi_k, N_u, L_{ui}: k<n, i\geq 1, |u| <
  n) \mbox{ for  $n\geq 1$}.
\end{eqnarray*}
For conditional probabilities and  expectations, we write:
\begin{eqnarray*}
 &&\P_{\xi, n}(\cdot ) = \P_\xi(\cdot | \D_n), \quad   \E_{\xi,n}(\cdot )= \E_\xi(\cdot | \D_n);\quad \P_{n}(\cdot )= \P(\cdot | \I_n), \quad  \E_{n}(\cdot )= \E(\cdot | \I_n);
 \\  && \P_{\xi, \mathscr{F}_n}(\cdot ) = \P_\xi(\cdot | \F_n), \quad   \E_{\xi,\F_n}(\cdot )= \E_\xi(\cdot | \F_n)  .
\end{eqnarray*}
As usual, we set $\N^* = \{1,2,3,\cdots \}$ and denote by
$$ U= \bigcup_{n=0}^{\infty} (\N^*)^n $$
the set of all finite sequences, where $(\N^*)^0=\{\varnothing \}$ contains the null sequence $ \varnothing$.

For all $u\in U$, let $\mathbb{T}(u)$ be the shifted tree of $\mathbb{T}$ at $u$  with defining elements $\{N_{uv}\}$: we have
1) $\varnothing \in \mathbb{T}(u)$, 2) $vi\in \mathbb{T}(u)\Rightarrow v\in \mathbb{T}(u)$ and  3) if  $v\in \mathbb{T}(u)$, then $vi\in \mathbb{T}(u)$ if and only if $1\leq i\leq N_{uv} $. Define $\mathbb{T}_n(u)=\{v\in \mathbb{T}(u): |v|=n\}$. Then  $\mathbb{T}=\mathbb{T}(\varnothing)$ and $\mathbb{T}_n=\mathbb{T}_n(\varnothing)$.

For every integer $m\geq 0$,  let  $H_m$ be the Chebyshev-Hermite polynomial of degree $m$ (\cite{Petrov75}):
\begin{equation}\label{eqCH}
   H_{m}(x)=m! \sum_{k=0}^{\lfloor \frac{m}{2}\rfloor} \frac{(-1)^k x^{m-2k}}{ k!(m-2k)! 2^k}.
\end{equation}
The first few Chebyshev-Hermite polynomials relevant to us are:
\begin{eqnarray*}
   & &  H_0(x)=1,  \\
    & & H_1(x) =x,\\
   & &  H_2(x)=x^2-1,\\
  &&  H_3(x)= x^3-3x,  \\
    & & H_4(x)= x^4-6x^2+3, \\
   && H_5 (x)= x^5-10x^3+15x, \\
    && H_6(x)= x^6-15x^4+45x^2-15,  \\
    &&   H_7(x)= x^7-21x^5+105x^3-105x, \\
    &&   H_8(x)= x^8-28x^6+210x^4-420x^2+105.
\end{eqnarray*}It is known that (\cite{Petrov75}) :    for every integer $m\geq 0$
\begin{equation*}
 \Phi^{(m+1)}(x) = \frac{d^{m+1}}{dx^{m+1}}\Phi(x)= (-1)^m\phi(x)H_m(x).
\end{equation*}

\subsection{Two preliminary  lemmas}
We first give an  elementary lemma which will be often used in Section  \ref{sec5}.
\begin{lem}\label{lemma2-1}
 \begin{enumerate}
   \item[(a)]  For $x,y \geq 0$,
   \begin{equation}\label{cbrw2.1}
         \ln^+(x+y) \leq 1+\ln^+x+ \ln^+y , \qquad  \ln(1+x) \leq 1+\ln^+x.
      \end{equation}
 \item[(b)] For each $\lambda >0$, there exists a constant $K_\lambda>0$, such that   \begin{equation}\label{cbrw3.4}
(\ln^+ x) ^{1+\lambda} \leq   K_\lambda  x,  \ \  x>0,
  \end{equation}
\item[(c)] For each $\lambda >0$, the function
 \begin{equation}\label{cbrw2.3}
  ( \ln(e^\lambda+ x) )^{1+\lambda} \;\; \mbox{ is concave for } \;  x>0.
  \end{equation}
  \end{enumerate}
\end{lem}
\begin{proof}
 Part (a) holds since    $\ln^+(x+y) \leq \ln^+(2 \max\{x,y\}) \leq  1+\ln^+x+ \ln^+y$.
Parts (b) and  (c)  can be verified easily.
\end{proof}

We next present
 the Edgeworth expansion  for sums of independent random variables, that we shall need  in Sections  \ref{sec3} and \ref{sec4}   to prove the main theorems.    Let us recall the theorem used in this paper  obtained by Bai and Zhao(1986, \cite{BaiZhao1986}),  that generalizing the case for i.i.d random variables (cf. \cite[P.159, Theorem 1]{Petrov75}).

Let $\{X_j\}$ be  independent random variables,  s atisfying for each $j\geq 1$
\begin{equation}\label{cbrwa1}
   \E X_j=0, \E |X_j|^{k} <\infty   \mbox{  with some integer   }  k \geq 3.
\end{equation}
We write  $B_n^2 = \sum_{j=1}^{n} \E X_j^2$ and only consider the nontrivial case $B_n>0$.
Let $\gamma_{\nu j}$  be  the $\nu$-order cumulant of $X_j$  for each $j\geq1$.
Write
\begin{align*}
    &  \lambda_{\nu,n }= n^{(\nu-2)/2} B_n^{-\nu}  \sum_{j=1}^n \gamma_{\nu j},  \quad {\nu=3,4\cdots, k}; \\
    &  Q_{\nu,n}(x)=  \sum { }^{'}(-1)^{\nu+2s}\Phi^{(\nu+2s)}(x) \prod_{m=1}^{\nu} \frac{1}{k_m!} \bigg(\frac{\lambda_{m+2,n}}{(m+2)!}\bigg)^{k_m}
    \\ & \qquad \quad =- \phi(x)\sum { }^{'}  H_{\nu+2s-1}(x)\prod_{m=1}^{\nu} \frac{1}{k_m!} \bigg(\frac{\lambda_{m+2,n}}{(m+2)!}\bigg)^{k_m},
\end{align*}
where the summation  $ \sum { }^{'}$  is carried out over all  nonnegative integer solutions $(k_1, \dots, k_\nu  )$  of the equations:
\begin{equation*}
   k_1+\cdots+  k_\nu=s \quad   \mbox{ and  } \quad  k_1+2k_2+\cdots +\nu k_{\nu}=\nu.
\end{equation*}
 For   $ 1\leq j\leq n$   and   $x\in \R$, define
\begin{align*}
    & F_n(x)= \P \Big (  {B_n}^{-1}  \sum_{j=1}^n X_j  \leq x\Big),    \quad v_j(t) = \E e^{itX_j}; \\
    &Y_{nj}= X_j \mathbf{1}_{\{ |X_j| \leq B_n\}},  \quad    Z_{nj}^{(x)}= X_{j }\mathbf{1}_{ \{|X_j| \leq B_n(1+|x|)\}}, \quad W_{nj}^{(x)}= X_{j }\mathbf{1}_{ \{|X_j| > B_n(1+|x|)\}}.
\end{align*}
The  Edgeworth expansion  theorem can be stated as follows.
\vspace{2mm}

\begin{lem} [\cite{BaiZhao1986}] \label{lem-Edge-exp}
  Let $n\geq 1$ and $X_1, \cdots, X_n$ be a sequence of  independent random variables satisfying    \eqref{cbrwa1}  and $ B_n>0$.   Then for the  integer $k\geq 3$,
\begin{multline*}
   | F_n(x) - \Phi(x)- \sum_{\nu=1}^{k-2} Q_{\nu n}(x)n^{-1/2} | \leq   C(k)\Bigg\{   (1+|x|)^{ -k} B_n^{-k} \sum_{j=1}^n  \E |W_{nj}^{(x)}|^k +    \\   (1+|x|)^{ -k-1} B_n^{-k-1} \sum_{j=1}^n\E |Z_{nj}^{(x)}|^{k+1}  +   (1+|x|)^{ -k-1} n^{k(k+1)/2}\Big( \sup_{|t|\geq \delta_n} \frac{1}{n} \sum_{j=1}^n |v_{j}(t)| +\frac{1}{2n} \Big)^n \Bigg\},
   \end{multline*}
where $\displaystyle \delta_n =  \frac{1}{12}  {B_n^2}{  (\sum_{j=1}^n\E  |Y_{nj}|^3)^{-1}  }$,  $C(k)>0 $  is a constant depending only on $k$.
 \end{lem}


\section{Convergence of the martingales $\{(N_{1,n},\D_n)\}$ and $\{( N_{2,n}, \D_n )\}$}\label{sec5}
Now we can proceed to prove the convergence of the two martingales defined in Section \ref{sec2}.

\subsection{Convergence of the martingale $ \{ (N_{1,n},\D_n) \}$ }
The fact that $ \{ (N_{1,n},\D_n) \}$ is a martingale can be easily shown: it suffices to notice that
 \begin{eqnarray*}
     \E_{\xi, n} {N_{1,n+1} }&=&\E_{\xi, n} \bigg(   \frac{1}{\Pi_{n+1}} \sum_{u\in \T_{n+1}} S_u  \bigg) = \frac{1}{\Pi_{n+1}} \E_{\xi, n} \bigg(   \sum_{u\in \T_{n}} \sum_{i=1}^{N_u}( S_u+ L_{ui} ) \bigg)  \\
     &=&\frac{1}{\Pi_{n+1}} \sum_{u\in \T_{n}} \E_{\xi, n} \Bigg(  \sum_{i=1}^{N_u}( S_u+ L_{ui} ) \Bigg)   \\
     &=&\frac{1}{\Pi_{n+1}} \sum_{u\in \T_{n}} m_n S_u=  N_{1,n}.
 \end{eqnarray*}
 We shall prove the convergence of the martingale by showing that the series
\begin{equation} \label{eqn-series1}
\sum_{n=1}^{\infty} I_n  \;\;  \mbox{ converges a.s., \; with }  \;\; I_{n}= N_{1,n+1}-N_{1,n}.
\end{equation}
To this end, we first establish a lemma. For $n\geq 1$ and $|u|= n$, set
\begin{eqnarray}\label{eqn-Xu}
   && X_u=S_u \bigg(\frac{N_u}{m_{|u|}} -1\bigg) + \sum_{i=1}^{N_u}\frac{L_{ui}}{m_{|u|}},
\end{eqnarray}
 and let $\widehat{X}_n $ be a generic random variable  of $X_u$,  i.e. $\widehat{X}_n $  has the same distribution with $X_u$ (for $|u|=n$). Recall
 that  $\widehat{N}_n $  has the same distribution as $N_u$,  $|u|=n$.

We proceed  the proof by proving the following lemma:
\begin{lem}\label{lem6}  Under the conditions of  Proposition \ref{th1a}, we have
\begin{equation}\label{cbrweq5-1}
 \E_{\xi} {|\widehat{X}_n|} (\ln^+{{|\widehat{X}_n|}} )^{1+\lambda} \leq K_{\xi}n \left( (\ln n)^{1+\lambda} +  \E_\xi  \frac{\widehat{N}_n}{ m_n} (\ln^+ \widehat{N}_n)^{1+\lambda} + (\ln^- m_n)^{1+\lambda}\right),
\end{equation}
where $K_\xi$ is a constant.
 \end{lem}
\begin{proof}

   For $u\in \T_n$,
\begin{eqnarray*}
  & &|X_u|\leq  |S_u|\left(1+ \frac{N_u}{m_{n}}\right) + \frac{ \abs{\sum_{i=1}^{N_u} L_{ui}}}{m_{n}}, \\
 & & \ln^+ |X_u|\leq 2+ \ln^+|S_u| + \ln(1+ N_u/m_n) + \ln^+ \abs{\frac{1}{m_{n}} \sum_{i=1}^{N_u} L_{ui} }, \\
 & &  4^{-\lambda}(\ln^+ |X_u|)^{1+\lambda}\leq  2^{1+\lambda}+( \ln^+|S_u|)^{1+\lambda} + \bigg(\ln\left(1+ \frac{N_u}{m_{n}}\right)\bigg)^{1+\lambda} + \bigg(\ln^+ \abs{ \frac{1}{m_{n}} \sum_{i=1}^{N_u} L_{ui} }\bigg)^{1+\lambda}.
\end{eqnarray*}
Hence we get that
 $$  4^{-\lambda} |X_u| (\ln^+|X_u|)^{1+\lambda}\leq \sum_{i=1}^8 \mathbb{J}_i,  $$
with
\begin{eqnarray*}
&& \mathbb{J}_1= 2^{1+\lambda} |S_u|\left(1+ \frac{N_u}{m_{n}}\right) ,\qquad   \mathbb{J}_2=  |S_u|( \ln^+|S_u|)^{1+\lambda} \left(1+ \frac{N_u}{m_{n}}\right), \\ &&  \mathbb{J}_3= |S_u|\left(1+ \frac{N_u}{m_{n}}\right)\bigg(\ln\left(1+ \frac{N_u}{m_{n}}\right)\bigg)^{1+\lambda}, \qquad  \mathbb{J}_4= |S_u|\left(1+ \frac{N_u}{m_{n}}\right) \bigg(\ln^+ \abs{ \frac{1}{m_{n}}\sum_{i=1}^{N_u} L_{ui} }\bigg)^{1+\lambda}, \\
   &&\mathbb{J}_5=   \frac{2^{1+\lambda}}{m_{n}}\abs{\sum_{i=1}^{N_u} L_{ui}} ,\quad  \mathbb{J}_6= \frac{( \ln^+|S_u|)^{1+\lambda}}{m_{n}} \abs{\sum_{i=1}^{N_u} L_{ui}},\quad  \mathbb{J}_7=\bigg(\ln\left(1+ \frac{N_u}{m_{n}}\right)\bigg)^{1+\lambda} \abs{\frac{1}{m_{n}}\sum_{i=1}^{N_u} L_{ui}}, \\ && \mathbb{J}_8=\frac{1}{m_{n}}\abs{\sum_{i=1}^{N_u} L_{ui}}\bigg(\ln^+ \abs{ \frac{1}{m_{n}} \sum_{i=1}^{N_u} L_{ui} }\bigg)^{1+\lambda}.
\end{eqnarray*}
Since \begin{equation*}
\lim_{n\rightarrow \infty} \frac{1}{n} \sum_{j=1}^{n} \E_\xi|\widehat{L}_j|^{q}  =\E |\widehat{L}_1|^q <\infty, \quad q=1,2 ,
\end{equation*}
there exists  a constant $K_\xi < \infty $ depending only on $\xi$ such that for $n\geq 1$ and $|u|= n$,
\begin{equation}\label{cbrweq3.1}
 \E_\xi |\widehat{L}_n|\leq K_\xi n,  \quad  \E_\xi |S_u|\leq \sum_{j=1}^{n} \E_\xi|\widehat{L}_j| \leq K_\xi n, \quad  \E_\xi |S_u|^2=\sum_{j=1}^{n} \E_\xi|\widehat{L}_j|^2 \leq K_\xi n.
\end{equation}
By the definition of the model,  $S_u$, $N_u$ and  $L_{ui}$  are mutually independent  under $\P_\xi$. On the basis of the above estimates,  we have the following inequalities,  where $K_\xi $  is a constant depending  on $\xi$, whose value may be different from lines to lines: for $n\geq 1$ and $|u|= n$,
\begin{align*}
   & \E_\xi  \J_1 =    2^{1+\lambda} \E_\xi |S_u| \E_\xi \left(1+ \frac{N_u}{m_{|u|}}\right) \leq   K_\xi n; \\
   & \E_\xi  \J_2  \leq K_{\lambda} \E_\xi( |S_u|^2 +|S_u|) \leq    K_\xi n    \quad (\mbox{by } \eqref{cbrw3.4}); \\
   & \E_\xi  \J_3  \leq    \E_\xi|S_u| \E_\xi \left(1+ \frac{N_u}{m_{|u|}}\right)\bigg(\ln\left(1+ \frac{N_u}{m_{n}}\right)\bigg)^{1+\lambda} \\
   &\qquad \leq K_\xi n \left(K_\xi+\E_\xi  \frac{\widehat{N}_n}{ m_n} (\ln^+ \widehat{N}_n)^{1+\lambda} +   \big(\ln^{-} m_n \big)^{1+\lambda} \right);
   \\ &
   \E_\xi \J_4 \leq \E_\xi |S_u| \; \E_\xi \left[  \left(1+ \frac{N_u}{m_{|u|}} \right) \bigg(\ln \Big(e^{\lambda}+ \frac{1}{m_{|u|}}\abs{ \sum_{i=1}^{N_u} L_{ui} } \Big)\bigg)^{1+\lambda} \right]
 \\ &  \qquad\leq (K_\xi n) \E_\xi \left[ \left(1+ \frac{N_u}{m_{|u|}} \right)   \bigg(\ln  \E_\xi \Big( e^{\lambda}+ \frac{1}{m_{|u|}}\sum_{i=1}^{N_u} \abs{ L_{ui} } \; \Big | \;  N_u\Big) \bigg)^{1+\lambda} \right]
 \\ & \qquad ~~~~ \mbox{(by Jensen's inequality under  $\E_\xi ( \cdot | N_u)$ using  the concavity of   $(\ln (e^{\lambda}+x))^{1+\lambda}$)}
 \\& \qquad =  (K_\xi n) \E_\xi \left(1+ \frac{N_u}{m_{|u|}} \right)   \bigg(\ln  \Big( e^{\lambda}+ \frac{1}{m_{|u|}}\sum_{i=1}^{N_u}\E_\xi \abs{ L_{ui} }\Big)\bigg)^{1+\lambda}
 \\ &  \qquad\leq K_\xi n \left(  K_\xi (\ln n)^{1+\lambda}+  \E_\xi \Big(\frac{1}{ m_{|u|} }N_u\big(\ln^+   { { N_u} }  \big)^{1+\lambda}\Big)  +2 \big(\ln^{-} m_n \big)^{1+\lambda}\right )
   \\ &\qquad\leq K_\xi n ( \ln n )^{1+\lambda}+ K_\xi n \E_\xi  \frac{1}{ m_n}\widehat{N}_n \big(\ln^+   {{ \widehat{N}_n }}  \big) ^{1+\lambda} + K_\xi n\big(\ln^{-} m_n \big)^{1+\lambda} ;
   \\ & \E_\xi  \J_5 \leq 2^{1+\lambda} \E_\xi |\widehat{L}_n|\leq  K_\xi n ;
   \\ &  \E_\xi  \J_6 = \E_\xi (\ln^+ |S_u|)^{1+\lambda}  \E_\xi \frac{1}{m_{|u|}} \abs{\sum_{i=1}^{N_u} L_{ui} }\leq  \E_\xi (\ln (e^{\lambda}+ |S_u|))^{1+\lambda}  \E_\xi \frac{1}{m_{|u|}} \abs{\sum_{i=1}^{N_u} L_{ui} }
    \\ &\qquad \leq  (\ln(e^{\lambda}+\E_\xi|S_u|) )^{1+\lambda} \E_\xi|\widehat{L}_n|\leq (\ln (K_\xi n) )^{1+\lambda} K_\xi n \leq K_\xi n (\ln n )^{1+\lambda};
   \\ & \E_\xi \J_7 \leq \E_\xi  \left[  \frac{1}{m_n} \sum_{i=1}^{N_u} (\E_\xi |L_{ui}|)   \Big(\ln\big (1+\frac{N_u}{m_n}\big)\Big)^{1+\lambda} \right]   \; \mbox{ (by the independence between $N_u$ and $L_{ui}$)}\\
    & \qquad \leq K_\xi n  \E_\xi\left[\frac{1}{m_n}N_u  3^{\lambda}
        \Big( 1+ (\ln^+ N_u)^{1+\lambda} + (\ln^-m_n)^{1+\lambda}\Big)\right]
         \\ & \qquad   \leq   K_\xi n + K_\xi n \E_\xi  \frac{1}{ m_n}\widehat{N}_n \big(\ln^+   {{ \widehat{N}_n }}  \big) ^{1+\lambda} + K_\xi n\big(\ln^{-} m_n \big)^{1+\lambda};
   \\ &  \E_\xi \J_8 \leq   \E_\xi \left[\frac{1}{m_n}   \Big|{\sum_{i=1}^{N_u} L_{ui}}\Big|  \bigg( \ln^+   \Big|{\sum_{i=1}^{N_u} L_{ui}}\Big|+ \ln^- m_n \bigg)^{1+\lambda}\right]
   \\ & \qquad \leq \E_\xi   \left[\frac{1}{m_n}   \Big|{\sum_{i=1}^{N_u} L_{ui}}\Big|     2^{\lambda} \bigg(   \Big(\ln^+   \Big|{\sum_{i=1}^{N_u} L_{ui}}\Big|\Big)^{1+\lambda} + (\ln^- m_n)^{1+\lambda} \bigg)         \right]
   \\ & \qquad \leq K_\lambda \frac{1}{m_n}\E_\xi  \Big|{\sum_{i=1}^{N_u} L_{ui}}\Big| ^2   +  2^{\lambda}  (\ln^- m_n)^{1+\lambda} \frac{1}{m_n} \E_\xi   \Big|{\sum_{i=1}^{N_u} L_{ui}}\Big|    \quad  (\mbox{  by  } \eqref{cbrw3.4})
   \\& \qquad \leq    K_\lambda \frac{1}{m_n}  \E_\xi \sum_{i=1}^{N_u} \E_\xi |L_{ui}|^2 +  2^{\lambda}  (\ln^- m_n)^{1+\lambda} \frac{1}{m_n} \E_\xi   {\sum_{i=1}^{N_u}\E_\xi \Big| L_{ui}}\Big|
   \\&\qquad\leq   K_\xi n + K_\xi n  (\ln^- m_n)^{1+\lambda}.
        \end{align*}
 Hence  we get that for  $n\geq 1$ and $|u|= n$,
 \begin{equation}\label{cbrweq3.4}
   \E_{\xi} {|X_u|} (\ln^+{{|X_u|}} )^{1+\lambda}\leq K_{\xi}n \left( (\ln n )^{1+\lambda} +  \E_\xi  \frac{\widehat{N}_n}{ m_n} \Big(\ln^+   {  \widehat{N}_n}   \Big) ^{1+ \lambda} + (\ln^- m_n)^{1+\lambda}\right).
 \end{equation}
This gives  \eqref{cbrweq5-1}.

\end{proof}

 \begin{proof}[Proof of Proposition \ref{th1a}]  We have already seen that $ \{ (N_{1,n},\D_n) \}$ is a martingale. We now prove its  convergence by showing the a.s. convergence of $\sum I_n $  (cf.  (\ref{eqn-series1})). Notice that
   $$   I_{n}= N_{1,n+1}-N_{1,n} = \frac{1 }{\Pi_n} \sum_{u\in \T_n} X_u.  $$
 We shall use a truncating argument to prove   the convergence.  Let
   $$
 X_u'= X_u \mathbf{1}_{\{|X_u| \leq \Pi_{|u|}\}}  \quad \mbox{ and }
 I_{n}'=  \frac{1 }{\Pi_n} \sum_{u\in \T_n} X'_u.
 $$
The following decomposition will play an important role:
\begin{equation} \label{eqn-decomposition}
  \sum_{n=0}^\infty I_n = \sum_{n=0}^\infty (I_n-I_n')+ \sum_{n=0}^\infty (I_n'-  \E_{\xi,\F_n} I_n' ) + \sum_{n=0}^\infty  \E_{\xi,\F_n} I_n'.
 \end{equation}
 We shall prove that each of the three series on the right hand side converges a.s.
To this end, let us first prove that
 \begin{equation}\label{cbrweq3.2}
  \sum_{n=1}^\infty\frac{1}{(\ln \Pi_n)^{1+\lambda}} \E_{\xi} {|\widehat{X}_n|} (\ln^+{{|\widehat{X}_n|}} )^{1+\lambda}<\infty \quad \mbox{ a.s. }
\end{equation}
Since $ \lim_{n\rightarrow\infty}  {\ln \Pi_n}/{ n} =\E\ln m_0 >0$ a.s., for a given constant $0 <\delta_1< \E\ln m_0$ and for $n$ large enough,
\begin{equation*}
  \ln \Pi_n > \delta_1 n,
\end{equation*}
so that, by Lemma   \ref{lem6},
\begin{equation*}
   \frac{1}{(\ln \Pi_n)^{1+\lambda}} \E_{\xi} {|\widehat{X}_n|} (\ln^+{{|\widehat{X}_n|}} )^{1+\lambda}\leq \frac{K_\xi}{\delta_1^{1+\lambda}} \frac{1}{n^{\lambda}}    \left[(\ln n)^{1+\lambda}+ \E_\xi \frac{\widehat{N}_n }{m_n} (\ln ^+  {\widehat{N}_n} )^{1+ \lambda} + (\ln^- m_n)^{1+\lambda}\right].
\end{equation*}
Observe that  for $ \lambda>1 $,  \begin{align*}
                                      & \E \sum_{n=1}^\infty \frac{1}{n^{\lambda} }\bigg[\E_\xi \frac{\widehat{N}_n }{m_n} (\ln ^+ {\widehat{N}_n} )^{1+ \lambda}+ (\ln^- m_n)^{1+\lambda}\bigg ]   \\
                                   = \  & \sum_{n=1}^\infty \frac{ 1}{n^{\lambda} } \bigg[   \E \frac{{\widehat{N}_0} }{m_0} (\ln^+  {\widehat{N}_0}  )^{1+\lambda}+   \E  (\ln^- m_0)^{1+\lambda}   \bigg]  <\infty,
                                  \end{align*}
which implies  that
\begin{equation*}
   \sum_{n=1}^\infty \frac{1}{n^{\lambda} }\bigg[\E_\xi \frac{\widehat{N}_n }{m_n} (\ln ^+ {\widehat{N}_n} )^{1+ \lambda}+ (\ln^- m_n)^{1+\lambda}\bigg ] <\infty  \mbox{~~ a.s.}
\end{equation*}
Therefore (\ref{cbrweq3.2}) holds.

%

For the first series $\sum_{n=0}^\infty (I_n-I_n')$ in (\ref{eqn-decomposition}),  we observe that
\begin{eqnarray*}
  \E_\xi |I_n-I_n'|  &=&  \E_\xi \abs{\frac{1}{\Pi_n}  \sum_{u\in \T_n} X_u \ind{|X_u| >\Pi_n} } \\
   &\leq &  \E_\xi \left\{\frac{1}{\Pi_n} \sum_{u\in \T_n}   \E_{\xi,\F_n} ({|X_u|}\ind{|X_u| >\Pi_n})\right\} \\
   &=&\E_{\xi}\big( {|\widehat{X}_n|} \ind{\abs{\widehat{X}_n} >\Pi_n}\big)\\ & \leq &\frac{1}{(\ln \Pi_n)^{1+\lambda}}\E_{\xi} {|\widehat{X}_n|} (\ln^+{{|\widehat{X}_n|}} )^{1+\lambda} .
\end{eqnarray*}
From this and  \eqref{cbrweq3.2}, \begin{equation*}
\E_{\xi}\sum_{n=0}^\infty \Big|I_n-I_n'\Big| \leq \sum_{n=0}^\infty \E_\xi |I_n-I_n'| <\infty,
\end{equation*}
whence   $\sum_{n=0}^\infty (I_n-I_n')$   converges a.s.

For the third series $\sum_{n=0}^{\infty} \E_{\xi,\F_n} I_n'$, as $ \E_{\xi,\F_n} I_n=0 $, we have
\begin{eqnarray*}
  &&\E_\xi\sum_{n=0}^{\infty}\abs{ \E_{\xi,\F_n} I_n' }  = E_\xi\sum_{n=0}^{\infty}\abs{ \E_{\xi,\F_n} (I_n-I_n') }
   \leq     \sum_{n=0}^{\infty}\E_{\xi} |I_n-I_n'| <\infty,
\end{eqnarray*}so that  $\sum_{n=0}^{\infty} \E_{\xi,\F_n} I_n'$ converges a.s.
It remains to
 prove that the second series
\begin{equation}\label{cbrweq3.3}
  \sum_{n=0}^\infty (I_n'- \E_{\xi,\F_n}I_n') \mbox{ converges a.s. }
\end{equation}
By the a.s. convergence of an  $L^2$ bounded martingale (see e.g. \cite[P. 251, Ex. 4.9]{Durrett96Proba}), we only need to show the convergence of the series
$   \sum_{n=0}^\infty \E_{\xi} (I_n'- \E_{\xi,\F_n}I_n')^2 .$
Notice
\begin{align*}
         \E_{\xi} (I_n'- \E_{\xi,\F_n}I_n')^2  &=  \E_\xi \Bigg(\frac{1}{\Pi_n}  \sum_{u\in\T_n } (X_u'- \E_{\xi,\F_n} X_u')\Bigg)^2
                  =  \E_{\xi} \Bigg(\frac{1}{\Pi_n^2}  \sum_{u\in\T_n }  \E_{\xi,\F_n} (X_u'- \E_{\xi,\F_n} X_u')^2\Bigg) \\
                 &\leq  \E_\xi \frac{1}{\Pi_n^2}  \sum_{u\in \T_n}  \E_{\xi,\F_n}  X_u'^2
               =\frac{1}{\Pi_n} \E_{\xi} (\widehat{X}_n^2\ind{ |\widehat{X}_n| \leq \Pi_n} ) \\
              &=   \frac{1}{\Pi_n} \E_{\xi} \Big( \widehat{X}_n^2\ind{ |\widehat{X}_n| \leq \Pi_n } \ind{  |\widehat{X}_n| \leq e^{2\lambda}} + \widehat{X}_n^2\ind{ |\widehat{X}_n| \leq \Pi_n } \ind{  |\widehat{X}_n| >e^{2\lambda}}  \Big)\\
               &\leq    \frac{e^{4\lambda}}{\Pi_n}+ \frac{1}{\Pi_n} \E_{\xi}\frac{ \widehat{X}_n^2 \Pi_n (\ln \Pi_n)^{-(1+\lambda) }  }{ | \widehat{X}_n| (\ln^+ |\widehat{X}_n|)^{-(1+\lambda)}}
               \\&
               \quad (\mbox{ because }  x(\ln x)^{-1-\lambda} \mbox{ is increasing for }  x > e^{2\lambda})  \\ &= \frac{e^{4\lambda}}{\Pi_n}+\frac{ 1 }{  (\ln \Pi_n)^{1+\lambda}}\E_{\xi}|\widehat{X}_n|(\ln^+ |\widehat{X}_n|)^{1+\lambda}.
            \end{align*}
Therefore  by \eqref{cbrweq3.2}, we see that  $\sum_{n=0}^\infty \E_{\xi} (I_n'- \E_{\xi,\F_n}I_n')^2 <\infty $ a.s..  This implies  \eqref{cbrweq3.3}.

\if
For convenience, we will use the following notation:
\begin{eqnarray*}
   && X_u= S_u \bigg(\frac{N_u}{m_{|u|}} -1\bigg) + \sum_{i=1}^{N_u}\frac{L_{ui}}{m_{|u|}},\quad X_u'= X_u \mathbf{1}_{\{|X_u| \leq \Pi_{|u|}, |L_{{ui} }|\leq |u|, 1\leq i\leq N_u\}};  \\
   &&  I_{n}= V_{n+1}-V_{n} = \frac{1 }{\Pi_n} \sum_{u\in \T_n} X_u,    \quad I_{n+1}'=  \frac{1 }{\Pi_n} \sum_{u\in \T_n} X'_u.
\end{eqnarray*}
Observe that,
\begin{equation*}
  \sum_{n=1}^{\infty} I_n =  \sum_{n=1}^{\infty} (I_n-I_n') +  \sum_{n=1}^{\infty} (I_n'- \E_{\xi, \F_n} I_n') +  \sum_{n=1}^{\infty}  \E_{\xi, \F_n} I_n'
\end{equation*}
So we will prove that the three series in the right hand side are finite a.s.

By the definition of the notations, we see that
 \begin{equation*}
    \sum_{n=1}^{\infty} (I_n-I_n') = \sum_{n=1}^{\infty} \frac{1}{\Pi_{n-1}} \sum_{u\in \T_{n-1}} X_u \mathbf{1}_{\{|X_u| > \Pi_{n-1}, | L_{{ui} }|\leq n-1, 1\leq i\leq N_u\}}.
 \end{equation*}
Then we have
      \begin{eqnarray*}
        \E_\xi |\sum_{n=1}^{\infty} (I_n-I_n')|  & \leq &\sum_{n=1}^{\infty} \E_\xi |(I_n-I_n')|  \\
          &\leq &  \sum_{n=1}^{\infty} \E_\xi  \bigg( \frac{1}{\Pi_n} \sum_{u\in\T_n} \E_{\xi, \F_n} X_u \mathbf{1}_{\{|X_u| >c^n\}}   \bigg).
      \end{eqnarray*}
Observe that

\begin{multline}\label{eq5}
  \P_\xi(I_n\neq I_n')  \leq   \E_\xi \bigg( \sum_{u\in \T_{n-1} } \sum_{i=1}^{N_u}  \P_\xi( |L_{ui}|>n)\bigg ) +  \\
   \E_\xi\bigg ( \sum_{u\in \T_{n-1} }  \P_{\xi, \I_n}( |X_u|>\Pi_{n-1}, |L_{ui}| \leq n-1, 1\leq i\leq N_u  ) \bigg).
\end{multline}

First we notice that \begin{eqnarray*}
              && \E \Big(\sum_{n=1}^{\infty}  \sum_{u\in \T_n} \P_\xi (|L_u|>n )\Big) \\
             &\leq& \E \Big(\sum_{n=1}^{\infty} \E_\xi (\sum_{u\in \T_n} \P_\xi (|L_u|>n ) )\Big )  \\
              &=& \E \Big(\sum_{n=1}^{\infty}\Pi_n \P_\xi (|H_n| >n) \Big)   \quad \mbox{(using the independence of $N_u$ and $ L_u$  under $\P_\xi$)}\\
              &\leq & \E \Big(\sum_{n=1}^{\infty}\Pi_n e^{-\delta_0 n} \E_\xi e^{\delta_0|H_n|} \Big)    \\
              &\leq & \KM   \sum_{n=1}^{\infty}e^{-\delta_0 n} \E \Pi_n < \Big (1-e^{(\ln \E m_0-\delta_0 )} \Big)^{-1}\KM  <\infty .
          \end{eqnarray*}
Then the first term in the right hand side of \eqref{eq5}  is summable.

Because $ \lim_{n\rightarrow +\infty} \ln \Pi_n /n =\E\ln m_0>0$,  then we have that for some constant  $c>1$, $n $ large enough,
\begin{equation*}
  \Pi_n>c^n.
\end{equation*}
Let $\widehat{X}_n$ be the generic random variable of $X_u \mathbf{1}_{\{|L_{ui}| \leq n, 1\leq i\leq N_u\}} $  and $\widehat{N}_n $ be the generic random variable of $N_u (|u|=n )$.
Then for $n$ large,
\begin{align*}
   & \E_\xi\bigg ( \sum_{u\in \T_{n-1} }  \P_{\xi, \I_n}( |X_u|>\Pi_{n-1}, |L_{ui}| \leq n-1, 1\leq i\leq N_u  ) \bigg) \\
   \leq~ &  \Pi_{n-1} \E_\xi\bigg (    \P_{\xi, \I_n}( |\widehat{X}_n|>\Pi_{n-1}  ) \bigg)
\\  \leq~ & \frac{\E_\xi |\widehat{X}_n| (\ln^{+}|\widehat{X}_n|)^{1+\lambda}}{(\ln \Pi_{n-1} )^{1+\lambda}}\\
\leq~ & \frac{1  }{ (n-1)^{1+\lambda} (\ln c)^{1+\lambda} }  (\E_\xi S_u^2 +2n) \E_\xi\frac{\widehat{N}_{n-1}}{m_{n-1}}(\ln^+|\frac{\widehat{N}_{n-1}}{m_{n-1}}|)^{1+\lambda}   \\
        \leq ~& \frac{(\KM+2)}{(n-1)^\lambda(\ln c)^{1+\lambda} }\E_\xi\frac{\widehat{N}_{n-1}}{m_{n-1}}(\ln^+|\frac{\widehat{N}_{n-1}}{m_{n-1}}|)^{1+\lambda}.
\end{align*}
Note that \begin{align*}
              &\E \sum_{n}   \frac{(\KM+2)}{(n-1)^\lambda\ln c}\E_\xi\frac{\widehat{N}_{n-1}}{m_{n-1}}(\ln^+|\frac{\widehat{N}_{n-1}}{m_{n-1}}|)^{1+\lambda} = \sum_{n}\frac{(\KM+2)}{n^\lambda\ln c}\E\frac{N}{m_0}(\ln^+|\frac{N}{m_0}|)^{1+\lambda}< \infty
          \end{align*}
implies  that a.s.
\begin{equation*}
   \sum_{n}   \frac{(\KM+2)}{(n-1)^\lambda\ln c}\E_\xi\frac{N_u}{m_{n-1}}(\ln^+|\frac{N_u}{m_{n-1}}|)^{1+\lambda}<\infty.
\end{equation*}
Therefore
\begin{equation*}
 \sum_n \E_\xi\bigg ( \sum_{u\in \T_{n-1} }  \P_{\xi, \I_n}( |X_u|>\Pi_{n-1}, |L_{ui}| \leq n-1, 1\leq i\leq N_u  ) \bigg) <\infty.
\end{equation*}

Now by the above arguments and  \eqref{eq5}, we will obtain that
\begin{equation*}
\sum_n \P_\xi(I_n\neq I_n')<\infty.
\end{equation*}
By the Borel-Cantelli Lemma, we get that $ \P_\xi (I_n\neq I_n'  \mbox{ infintey often} ) =0$ and hence that $  \sum_{n=1}^{\infty} (I_n-I_n')<\infty $.

Next we are going to prove that
\begin{equation*}
  \sum_{n=1}^\infty (I_n' -\E_{\xi, \I_n} I_n')<\infty \quad a.s.
\end{equation*}
It will follow from the fact that $ \sum_{n=1}^{\infty} \E_\xi (I_n' -\E_{\xi, \I_n} I_n')^2 <\infty$.
Notice that
\begin{eqnarray*}
 && \E_\xi (I_n' -\E_{\xi, \I_n} I_n')^2 = \E_\xi \bigg(  \E_{\xi, \I_n} \Big( \frac{1}{\Pi_{n-1}} \sum_{u\in \T_{n-1}} (X_u-\E_{\xi, \I_n} X_u' )  \Big)^2 \bigg)
 \\&=&
         \E_{\xi} \bigg(  \frac{1}{\Pi_{n-1}^2} \sum_{u\in\T_{n-1}}  \E_{\xi, \I_n}(X'_u-\E_{\xi, \I_n} X_u' )^2\bigg)
\\ & \leq&
      \frac{1}{\Pi_{n-1}^2} \E_{\xi} \bigg( \sum_{u\in\T_{n-1}}  \E_{\xi, \I_n}\big(X_u^2\mathbf{1}_{\{|X_u| < \Pi_{n-1}, |L_{{ui} }|\leq n, 1\leq i\leq N_u\}} \big) \bigg)
\\ &=&
      \frac{1}{\Pi_{n-1}}  \E_{\xi} \bigg( \E_{\xi, \I_n}\big(\widehat{X}_n^2\mathbf{1}_{\{|\widehat{X}_n| < \Pi_{n-1}\}} \big) \bigg)
\\ &=&
       \frac{1}{\Pi_{n-1}}  \E_{\xi} \bigg( \E_{\xi, \I_n}\big(\widehat{X}_n^2\mathbf{1}_{\{|\widehat{X}_n| < \Pi_{n-1}\}} (\mathbf{1}_{\{|\widehat{X}_n| \leq c\}}  + \mathbf{1}_{\{|\widehat{X}_n| > c\}})\big) \bigg)
\\ &\leq &
    \frac{c^2}{ \Pi_{n-1} } + \frac{1}{\Pi_{n-1}}  \E_\xi \bigg(  \E_{ \xi, \I_n}  \frac{\widehat{X}_n^2 \Pi_{n-1} (\ln^+ \Pi_{n-1})^{-(1+\lambda) }  }{ |\widehat{X}_n| (\ln^+ \widehat{X}_n )^{-(1+\lambda) }} \bigg)
\\ & \leq &  \frac{c^2}{ \Pi_{n-1} } +\frac{\E_\xi |\widehat{X}_n| (\ln^{+}|\widehat{X}_n|)^{1+\lambda}}{(\ln \Pi_{n-1} )^{1+\lambda}}
\end{eqnarray*}
Then by use the properties of $\Pi_n$, we see that $ \sum_{n=1}^{\infty} \E_\xi (I_n' -\E_{\xi, \I_n} I_n')^2 <\infty$.
Now we turn to prove that $\sum_{n=1}^\infty \E_{\xi, \I_n} I_n' <\infty $.

Observe that ( using that symmetry of  $L_{u}$),
\begin{eqnarray*}
   \E_\xi \bigg|\sum_{n=1}^\infty \E_{\xi, \I_n} I_n' \bigg|&=&  \E_\xi \Big|\sum_{n=1}^\infty \E_{\xi, \I_n}  \Big( \frac{1}{\Pi_{n-1}} \sum_{u\in \T_{n-1}} X_u\mathbf{1}_{\{|X_u|\leq \Pi_{n-1}, |L_{ui}| \leq n, 1\leq i \leq N_u \}  } \Big)  \Big| \\
   &=&  \E_\xi \Big|\sum_{n=1}^\infty \E_{\xi, \I_n}  \Big( \frac{1}{\Pi_{n-1}} \sum_{u\in \T_{n-1}} X_u\mathbf{1}_{\{|X_u|> \Pi_{n-1}, |L_{ui}| \leq n, 1\leq i \leq N_u   \}} \Big)  \Big| \\
   &\leq &  \sum_{n=1}^\infty \frac{\E_\xi |\widehat{X}_n| (\ln^{+}|\widehat{X}_n|)^{1+\lambda}}{(\ln \Pi_{n-1} )^{1+\lambda}}<\infty.
\end{eqnarray*}
Now the desired result follows.
\fi
Combining the above results, we see that the series $\sum I_n$ converges a.s., so that $N_{1,n}$ converges a.s. to  $$V_1=\sum_{n=1}^\infty (N_{1,n+1} -N_{1,n})+ N_{1,1}.$$
\end{proof}

\subsection{Convergence of the martingale $ \{(N_{2,n},\D_n)\}$}

  To see that $ \{(N_{2,n},\D_n)\}$ is a martingale, it suffices to notice that (remind that we have assumed  $\ell_n=0$)
 \begin{eqnarray*}
     \E_{\xi, n} {N_{2,n+1} }&=&\E_{\xi, n}(s_{n+1}^2W_{n+1})-   \E_{\xi, n} \bigg(   \frac{1}{\Pi_{n+1}} \sum_{u\in \T_{n+1}} S_u^2 \bigg)\\
      & = & s_{n+1}^2W_{n}-  \frac{1}{\Pi_{n+1}} \sum_{u\in \T_{n}} \E_{\xi, n} \bigg(  \sum_{i=1}^{N_u}(S_{u}+L_{ui})^2  \bigg)  \\
     &=&s_{n+1}^2W_{n}-  \frac{1}{\Pi_{n+1}} \sum_{u\in \T_{n}} \E_{\xi, n} \Bigg(  \sum_{i=1}^{N_u}( S_u^2+ 2S_u L_{ui}+L_{ui}^2 ) \Bigg)   \\
     &=&s_{n+1}^2W_{n}-  \frac{1}{\Pi_{n+1}} \sum_{u\in \T_{n}} \E_{\xi, n} \Bigg(  \sum_{i=1}^{N_u} \E_{\xi, n}\{ ( S_u^2+ 2S_u L_{ui}+L_{ui}^2 )|N_u \}\Bigg)\\
     &=&s_{n+1}^2W_{n}- \frac{1}{\Pi_{n+1}} \sum_{u\in \T_{n}}  m_n(S_u^2+\sigma_n^{(2)}) =s_n^2W_n- \frac{1}{\Pi_{n}} \sum_{u\in \T_{n}} S_u^2 =N_{2,n}.
 \end{eqnarray*}

As in the case of $ \{(N_{1,n},\D_n)\}$,  we will prove the convergence of the martingale $ \{(N_{2,n},\D_n)\}$  by showing that
\begin{equation*}
  \sum_{n=1}^{\infty} (N_{2,n+1} -N_{2,n}) \mbox{ converges  a.s., }
\end{equation*}
following the same lines as before.  For $n\geq 1$ and $|u|=n$,
we will still use the notation $X_u$ and $I_n$, but this time they are defined by:
 \begin{align} \label{Xu2}
    X_u&~= (S_u^2- s_n^2) (1-\frac{N_u}{m_n}) + \frac{1}{m_n} \sum_{i=1}^{N_u} (\sigma_n^{(2)}-L_{ui}^2) - \frac{2}{m_n} S_u \sum_{i=1}^{N_u}  L_{ui}, \\
    \label{In2}
    I_n&= N_{2,n+1}-N_{2,n} = \frac{1}{\Pi_{n}} \sum_{u\in \T_{n}} X_u.
 \end{align}
Instead of
 Lemma \ref{lem6}, we have:
\begin{lem}\label{lem7} For $n\geq 1$  and $|u|= n$,  let $\widehat{X}_n $ be a random variable with the common distribution  of $X_u$ defined by (\ref{Xu2}), under the law $\P_\xi$. If  the conditions of Proposition \ref{th2a} holds, then
  \begin{equation}\label{cbrweq5-2}
   \E_{\xi} {|\widehat{X}_n|} (\ln^+{{|\widehat{X}_n|}} )^{1+\lambda}\leq K_{\xi} n^2   \bigg[   \E_\xi  \frac{\widehat{N}_n}{ m_n} (\ln ^+  { \widehat{N}_n} )^{1+\lambda}+ (\ln^- m_n)^{1+\lambda}+1 \bigg].
  \end{equation}
\end{lem}
\begin{proof}
Observe that for $|u|= n$,
\begin{align*}
                &|X_u|~\leq |s_n^2-S_u^2|(1+\frac{N_u}{m_n})+ \left| \frac{1}{m_n} \sum_{i=1}^{N_u} (\sigma_n^{(2)} -L_{ui}^2)\right|+ |S_u| \left|\frac{2}{m_n}\sum_{i=1}^{N_u}L_{ui}\right|,  \\
               & \ln^+ |X_u|~ \leq 2+ \ln^+|s_n^2-S_u^2|+\ln (1+\frac{N_u}{m_n})+\ln^+ \left| \frac{1}{m_n} \sum_{i=1}^{N_u} (\sigma_n^{(2)} -L_{ui}^2) \right| \\ &\qquad \qquad  \qquad    +  \ln^+ \left|\frac{2}{m_n} \sum_{i=1}^{N_u}   L_{ui} \right|  +  \ln^+|S_u|,\\
               &  6^{-\lambda} (\ln^+|X_u|)^{1+\lambda} \leq 2^ {1+\lambda}+  (\ln^+|s_n^2-S_u^2|)^{1+\lambda}+ (\ln(1+\frac{N_u}{m_n}))^{1+\lambda} \\  &\qquad   \qquad   \qquad+  \left(\ln^+ \left| \frac{1}{m_n} \sum_{i=1}^{N_u} (\sigma_n^{(2)} -L_{ui}^2) \right|\right)^{1+\lambda} +  \left(\ln^+ \left|\frac{2}{m_n} \sum_{i=1}^{N_u}   L_{ui} \right| \right )^{1+\lambda}+( \ln^+|S_u| )^{1+\lambda}.
             \end{align*}
Therefore
    $$ 6^{-\lambda}|X_u| (\ln^+|X_u|)^{1+\lambda} \leq  \sum_{i=1}^8 \mathbb{K}_i $$
  with
  \begin{align*}
  & \K_1 =|s_n^2-S_u^2|  (1+\frac{N_u}{m_n})\Bigg[ 2^ {1+\lambda}+ \Big(\ln(1+\frac{N_u}{m_n})\Big)^{1+\lambda}+ \left(\ln^+ \left| \frac{1}{m_n} \sum_{i=1}^{N_u} (\sigma_n^{(2)} -L_{ui}^2) \right|\right)^{1+\lambda}\\ & \hspace{7.8cm} +  \left(\ln^+ \left|\frac{2}{m_n} \sum_{i=1}^{N_u}   L_{ui} \right| \right )^{1+\lambda} \Bigg],\\
  & \K_2 =|s_n^2-S_u^2|  (1+\frac{N_u}{m_n}) \bigg[ (\ln^+|s_n^2-S_u^2|)^{1+\lambda}+ ( \ln^+|S_u| )^{1+\lambda} \bigg],
  \\ & \K_3 = \left| \frac{1}{m_n} \sum_{i=1}^{N_u} (\sigma_n^{(2)} -L_{ui}^2) \right| \bigg[2^ {1+\lambda}+ (\ln^+|s_n^2-S_u^2|)^{1+\lambda}+ ( \ln^+|S_u| )^{1+\lambda} \bigg],\\ &
  \K_4 =  \left| \frac{1}{m_n} \sum_{i=1}^{N_u} (\sigma_n^{(2)} -L_{ui}^2) \right|
        \Big(\ln(1+\frac{N_u}{m_n})\Big)^{1+\lambda},
  \\ & \K_5=\left| \frac{1}{m_n} \sum_{i=1}^{N_u} (\sigma_n^{(2)} -L_{ui}^2) \right|
        \left[\left(\ln^+ \left|\frac{2}{m_n} \sum_{i=1}^{N_u}   L_{ui} \right| \right )^{1+\lambda}+ \left(\ln^+ \left| \frac{1}{m_n} \sum_{i=1}^{N_u} (\sigma_n^{(2)} -L_{ui}^2) \right|\right)^{1+\lambda}\right],
  \\ & \K_6 = \left|\frac{2}{m_n} \sum_{i=1}^{N_u}   L_{ui} \right| |S_u|
          \bigg[ 2^ {1+\lambda}+ (\ln^+|s_n^2-S_u^2|)^{1+\lambda}+ ( \ln^+|S_u| )^{1+\lambda} \bigg],\\
  & \K_7 =  \left|\frac{2}{m_n} \sum_{i=1}^{N_u}   L_{ui} \right| |S_u|\Big(\ln(1+\frac{N_u}{m_n})\Big)^{1+\lambda},
  \\ & \K_8= \left|\frac{2}{m_n} \sum_{i=1}^{N_u}   L_{ui} \right| |S_u|
            \left[\left(\ln^+ \left|\frac{2}{m_n} \sum_{i=1}^{N_u}   L_{ui} \right| \right )^{1+\lambda}+ \left(\ln^+ \left| \frac{1}{m_n} \sum_{i=1}^{N_u} (\sigma_n^{(2)} -L_{ui}^2) \right|\right)^{1+\lambda}\right].
 \end{align*}
It is  clear that \eqref{cbrweq3.1}  remains valid here; similarly, we get
\begin{equation*}
   \E_\xi |\sigma_n^{(2)} -L_{ui}^2| =  \E_\xi |\sigma_n^{(2)} -\widehat{L}_{n}^2|^2  \leq K_\xi n
\end{equation*}
(recall that $\widehat{L}_{n}$ is a random variable with the same distribution as $ L_{ui}$ for any $|u|=n $ and $i\geq 1$).
By the definition of the model,  $S_u$, $N_u$ and  $L_{ui}$  are mutually independent  under $\P_\xi$. On the basis of the above estimates,  we have the following inequalities: for  $|u|= n$,
\begin{align*}
  \E_\xi \K_1 & \leq \E_\xi| S_u^2+s_n^2| \E_\xi   (1+\frac{N_u}{m_n})\Bigg[ 2^ {1+\lambda}+ \Big(\ln(1+\frac{N_u}{m_n})\Big)^{1+\lambda}+\left(\ln \Big(e^{\lambda}+ \frac{1}{m_n} \sum_{i=1}^{N_u} \E_\xi|\sigma_n^{(2)} -L_{ui}^2|\Big) \right)^{1+\lambda} \\ & \qquad  +  \left(\ln\Big(e^{\lambda}+   \frac{2}{m_n} \sum_{i=1}^{N_u} \E_\xi | L_{ui}|   \Big)\right )^{1+\lambda} \Bigg]  \quad (\mbox{by Jensen's inequality under } \E_\xi(\cdot | N_u) )
  \\ & \leq K_\xi n\left[K_\xi + \E_\xi  \frac{N_u}{m_n}  \big( \ln^+{N_u} \big)^{1+\lambda}+    \big( \ln^-{m_n} \big)^{1+\lambda}+ (\ln n)^{1+\lambda}  \right];
 \\\E_\xi \K_2 & \leq 2(\E_\xi |S_u |^{2+\varepsilon} +   |s_n|^{2+\varepsilon} )\leq K_\xi n^{2};
 \\ \E_\xi \K_3 & \leq  \E_\xi  \Big(\frac{1}{m_n} \sum_{i=1}^{N_u} \E_\xi |\sigma_n^{(2)} -L_{ui}^2|\Big) \bigg(2^ {1+\lambda}+ (\ln(e^{\lambda}+\E_\xi|s_n^2-S_u^2|) )^{1+\lambda}+ ( \ln (e^{\lambda}+\E_\xi|S_u|) )^{1+\lambda} \bigg)\\ &
  \leq      K_\xi n (\ln n )^{1+\lambda};
  \\  \E_\xi \K_4 & \leq  K_\xi n+ K_\xi n \E_\xi  \frac{N_u}{m_n} \Big(\ln(1+\frac{N_u}{m_n})\Big)^{1+\lambda};
 \\ \E_\xi \K_5 &\leq 3^{\lambda} \E_\xi \frac{1}{m_n}\bigg| \sum_{i=1}^{N_u}  ( \sigma^{(2)}_n - L_{ui}^2)\bigg|   \Bigg[  \Big(  \ln^+\Big | \sum_{i=1}^{N_u}    L_{ui} \Big|\Big) ^{1+\lambda} +  \Big(  \ln^+\Big | \sum_{i=1}^{N_u}    (\sigma^{(2)}_n - L_{ui}^2) \Big|\Big) ^{1+\lambda} + \\ &  \qquad \qquad \qquad\qquad \qquad \qquad  2 (\ln^- m_n)^{1+\lambda} + 1    \Bigg]
 \\ & \leq  K_\lambda \frac{1}{m_n} \E_\xi \Bigg[\bigg| \sum_{i=1}^{N_u}  ( \sigma^{(2)}_n - L_{ui}^2)\bigg| ^2+  \Big(  \ln^+\Big | \sum_{i=1}^{N_u}    L_{ui} \Big|\Big) ^{2+2\lambda}  \Bigg]+
 \\ &\qquad        K_\lambda \E_\xi  \frac{1}{m_n}\bigg| \sum_{i=1}^{N_u}  ( \sigma^{(2)}_n - L_{ui}^2)\bigg|^2 + K_\xi n   (  (\ln^- m_n)^{1+\lambda}+1)  \qquad (\mbox{by \eqref{cbrw3.4} and }  2ab \leq  a^2+b^2)
 \\ & \leq_{\eqref{cbrw3.4}} K_\lambda \frac{1}{m_n} \E_\xi  \Big[ \sum_{i=1}^{N_u} \E_\xi ( \sigma^{(2)}_n - L_{ui}^2)^2\Big] +   K_\lambda \frac{1}{m_n} \E_\xi  \Big[ \sum_{i=1}^{N_u}    \E_\xi\Big |L_{ui} \Big| \Big]+K_\xi n   (  (\ln^- m_n)^{1+\lambda}+1)
 \\ & \leq  K_\xi n \big (  (\ln^- m_n)^{1+\lambda}+1\big);
  \\  \E_\xi \K_6 & \leq_{\eqref{cbrw2.1} }  \E_\xi \bigg(\frac{2}{m_n} \sum_{i=1}^{N_u}   \E_\xi |L_{ui}|\bigg) \E_\xi  \Big[K_\lambda |S_u|\big(1+  (\ln^+|S_u|)^{1+\lambda}+   (\ln s_n^2)^{1+\lambda} \big)\Big]
  \\ &\leq _{\eqref{cbrw3.4} }  K_\xi n  \E_\xi  \Big[ |S_u|^2+ |S_u|+      s_n^2  \big)\Big] \leq K_\xi n^2  ;
  \\ \E_\xi \K_7 & \leq \E_\xi  \left(\Big(\ln(1+\frac{N_u}{m_n})\Big)^{1+\lambda}\frac{2}{m_n} \sum_{i=1}^{N_u} \E_\xi  |L_{ui}|\right)  \E_\xi |S_u|
  \\ & \leq K_\xi n^2 \Big[ \E_\xi  \frac{N_u}{m_n}  \Big(\ln ^+ {N_u}  \Big)^{1+\lambda}  + (\ln^- m_n)^{1+\lambda}\Big];
  \\  \E_\xi \K_8 &  \leq  K_\xi n^2 \big (  (\ln^- m_n)^{1+\lambda}+1\big)
 \qquad ( \mbox{similar reason as  in the estimation for } \E_\xi \mathbb{K}_5 ) .
\end{align*}
Combining the above estimates,  we get that
 \begin{equation}\label{cbrweq4.1}
   \E_{\xi} {|\widehat{X}_n|} (\ln^+{{|\widehat{X}_n|}} )^{1+\lambda}  \leq  K_\xi n^2 \bigg ( \E_\xi  \frac{\widehat{N}_n}{ m_n}\Big (\ln^+ {\widehat{N}_n}\Big)^{1+\lambda}+  (\ln^- m_n)^{1+\lambda}+1\bigg)
 \end{equation}
This ends the proof of Lemma \ref{lem7}.
\end{proof}

\begin{proof}[Proof of Proposition \ref{th2a} ]
The proof is almost the same  as that of Proposition \ref{th1a}:  we still use the decomposition (\ref{eqn-decomposition}), but with $I_n$ and $X_u$ defined by
(\ref{In2}) and (\ref{Xu2}), and Lemma \ref{lem7} instead of
  Lemma \ref{lem6}, to prove that the series  $ \sum_{n=0}^\infty  (N_{2,n+1}-N_{2,n}) $ converges  a.s., yielding that
$\{N_{2,n}\}$  converges a.s. to  $$V_2=\sum_{n=1}^\infty (N_{2,n+1}-N_{2,n})+ N_{2,1}.$$
\end{proof}

\section{Proof of  Theorem \ref{th1} }\label{sec3}


\if
For each $n$,
we choose an integer $k_n<n$ as follows.
Let  $\beta$ with $1/2<\beta<1$  and   $\alpha>2/
   (\beta^{-1}-1)$. For $j\in \mathbb{N}$ and $j^{ \alpha/\beta} \leq
   n<(j+1)^{\alpha/\beta}$, set $k_n=a_j=[j^{\alpha}]$. Let $t_n= \ell_n+s_n t$ for $t\in \mathbb{R}$ and $n\geq 1$.
   Then by \eqref{cbrweq11},
\begin{equation}\label{cbrweq12}
    W_n( s_n t )=\mathds{A}_n+\mathds{B}_n+ W_{k_n}\Phi(t),
\end{equation}
where
\begin{eqnarray*}
  \mathds{A}_n &=& \frac{1}{\pi_{k_n }}   \sum_{u\in \mathbb{T}_{k_n}}
  \left\{W_{n-k_n}(u,s_n t-S_u)-  {\E}_{\xi,k_n} (W_{n-k_n}(u,s_n t-S_u) )\right\},\\
   \mathds{B}_n&=&\frac{1}{\pi_{k_n }}   \sum_{u\in \mathbb{T}_{k_n}} [{\E}_{\xi,k_n} (W_{n-k_n}(u,s_n t-S_u) )
   -\Phi (t)].
   \end{eqnarray*}
Here we remind that the random variables $W_{n-k_n}(u,s_n t-S_u)$ are independent of each other under the conditional probability  $\P_{\xi, k_n}$.

For the proof of Theorem \ref{th1}, as in Asmussen  and Kaplan \cite{AsmussenKaplan76BRW2}, the main idea is to use the decomposition formula \eqref{cbrweq12}, proving that $\mathds{A}_n\rightarrow 0$ and $\mathds{B}_n\rightarrow 0$ in probability. However, new ideas will be needed for the proof of the later due to the appearance of the random environment.
\fi

\subsection{A key decomposition }

 For $u\in (\N^*)^k (k\geq 0)
 $  and $n\geq 1$,
write for $B\subset \R$,
\begin{equation*}
Z_{n}(u,B)= \sum_{v \in \mathbb{T}_n(u)}\mathbf{1}_B(S_{uv}-S_u).
\end{equation*}
It can be easily seen that the law of $Z_{n}(u,B)$ under ${\P}_{\xi}$ is the same as that of $Z_n(B)$ under ${P}_{\theta^k\xi}$. Define
   \begin{eqnarray*}
  &&W_{n}(u,B) =Z_{n}(u,B)/\Pi_n(\theta^k\xi), \quad  W_n(u,t) = W_n(u,(-\infty,t]), \\
  && W_{n}(B) =Z_{n}( B)/ \Pi_n, \quad W_n(t) =W_n((-\infty,t]).
\end{eqnarray*}
By definition, we have $\Pi_{n}(\theta^k\xi)=m_k\cdots m_{k+n-1}$, $Z_n(B)= Z_n(\varnothing, B)$, $W_n(B)=W_n(\varnothing,B)$, $W_n= W_n(\mathbb{R})$.
 The following decomposition will play a key role in our approach: for $k\leq n$,
\begin{equation}\label{cbrweq11}
    Z_n(B)=\sum_{u\in \mathbb{T}_k} Z_{n-k}(u, B-S_u).
\end{equation}
   Remark that by our definition, for $u\in \T_k $, $$Z_{n-k}(u,B-S_u)=\sum_{v_1\cdots v_{n-k} \in \mathbb{T}_{n-k}(u) } \mathbf{1}_B(S_{uv_1\cdots v_{n-k}})$$ represents number of the descendants of $u$ at time $n$ situated in  $B$.

For each $n$,
we choose an integer $k_n<n$ as follows. Let $\beta$ be a real number  such that
$
   \max{\{\frac{2}{\lambda}, \frac{3}{\eta}\}}<\beta<\frac{1}{4}
$
and  set $k_n=\lfloor n^{\beta}\rfloor$, the integral part of $n^{\beta}$.
   Then on the basis of  \eqref{cbrweq11},  the following  decomposition will hold:
\begin{equation}\label{cbrweq12}
  \Pi_n^{-1} Z_n(s_n t)  - \Phi(t) W=\mathds{A}_n+\mathds{B}_n+ \mathds{C}_n,
\end{equation}
where
\begin{eqnarray*}
  \mathds{A}_n &=& \frac{1}{\Pi_{k_n }}   \sum_{u\in \mathbb{T}_{k_n}}
  \left[W_{n-k_n}(u,s_n t-S_u)-  {\E}_{\xi,k_n} W_{n-k_n}(u,s_n t-S_u) \right],\\
   \mathds{B}_n&=&\frac{1}{\Pi_{k_n }}   \sum_{u\in \mathbb{T}_{k_n}}\left [{\E}_{\xi,k_n} W_{n-k_n}(u,s_n t-S_u)
   -\Phi (t)\right],
   \\ \mathds{C}_n&=&   (W_{k_n}-W)\Phi(t).
   \end{eqnarray*}
Here we remind that the random variables $W_{n-k_n}(u,s_n t-S_u)$ are independent of each other under the conditional probability  $\P_{\xi, k_n}$.

\subsection{Proof of Theorem \ref{th1}}
 First,  observe that the condition $ \E m_0^{-\delta}<\infty$ implies that $\E \big(\ln^- m_0\big)^{ \kappa }<\infty$  for all $\kappa>0$.
 So the hypotheses of  Propositions \ref{th1a} and \ref{th2a}  are satisfied under the conditions of Theorem \ref{th1}.

By virtue of the decomposition \eqref{cbrweq12}, we shall divide  the proof  into three lemmas.
\begin{lem}\label{lem1}
Under the hypothesis of  Theorem \ref{th1},
\begin{equation}\label{eq6}
 \sqrt{n}\mathds{A}_n \xrightarrow{n \rightarrow \infty } 0  \mbox{ a.s.}
\end{equation}
\end{lem}
\begin{lem}\label{lem2} Under the hypothesis of  Theorem \ref{th1},
\begin{equation}\label{eq7}
  \sqrt{n} \mathds{B}_n  \xrightarrow{n \rightarrow \infty } \frac{1}{6}{\E \sigma_0^{(3)}}{(\E \sigma_0^{(2)})^{-\frac{3}{2}} }(1-t^2)\phi(t) W-(\E \sigma_0^{(2)})^{-\frac{1}{2}}\phi(t) \, V_1  \mbox{ a.s. }
\end{equation}
\end{lem}
\begin{lem}\label{lem3} Under the hypothesis of  Theorem \ref{th1},
\begin{equation}\label{eq8}
\sqrt{n}\mathds{C}_n \xrightarrow{n \rightarrow \infty } 0 \mbox{ a.s. }
\end{equation}
\end{lem}

Now we go to prove the lemmas subsequently.

\begin{proof}[Proof of  Lemma \ref{lem1}]
For ease of  notation, we define for $|u|=k_n$,
\begin{align*}
    &  X_{n,u}= W_{n-k_n}(u,s_n t-S_u) - \E_{\xi, k_n}W_{n-k_n}(u,s_n t-S_u), ~~\bar{X}_{n,u} =  X_{n,u} \mathbf{1}_{\{|X_{n,u}| <\Pi_{k_n}\}}, \\
     & \bar{A}_{n} =   \frac{1}{\Pi_{k_n}} \sum_{u\in \T_{k_n} } \bar{X}_{n,u}.
\end{align*}
Then we see that $ |X_{n,u}|\leq W_{n-k_n}(u)+1$.

To prove Lemma \ref{lem1}, we will use the extended Borel-Cantelli Lemma. We can obtain the required result once we prove that $\forall \varepsilon>0$,
\begin{equation}\label{eq10}
  \sum_{n=1}^{\infty} \P_{k_n} (|\sqrt{n} A_n| >2 \varepsilon) <\infty.
\end{equation}
Notice that \begin{eqnarray*}
              && \P_{k_n}(|A_n| >2\frac{\varepsilon}{\sqrt{n}})    \\
              &\leq&  \P_{k_n} (A_n\neq \bar{A}_n ) +\P_{k_n} (|\bar{A}_n - \E_{\xi, k_n} \bar{A}_n| > \frac{\varepsilon}{\sqrt{n}} )  +\P_{k_n} ( |\E_{\xi, k_n} \bar{A}_n| >\frac{\varepsilon}{\sqrt{n}}).
            \end{eqnarray*}
We will proceed the proof in 3 steps.\\
{\bf Step 1 } We first prove that
\begin{equation}\label{cbrweq3-7}
 \sum_{n=1}^{\infty}{\P}_{k_n} (A_n\neq \overline{A}_n) <\infty.
\end{equation}
To this end,  define  $$W^*=\sup_n W_n,$$
and we need the following result  :
\begin{lem}(\cite[Th.  1.2]{LiangLiu10})\label{lem5}
  Assume    \eqref{cbrweq1}  for some $\lambda>0$  and $\E m_0^{-\delta}<\infty $  for some $\delta>0$.  Then
  \begin{equation}\label{cbrweq17a}
    {\E}(W^*+1)(\ln (W^*+1))^{\lambda} <\infty.
\end{equation}
\end{lem}
We observe that
\begin{align*}
      \P_{k_n}(A_n\neq \overline{A}_n )&\leq \sum_{u\in \T_{k_n} } \P_{k_n}(X_{n,u}\neq \overline{X}_{n,u} )  =  \sum_{u\in \T_{k_n} } \P_{k_n}(|X_{n,u}|\geq \Pi_{k_n}) \\&\leq \sum_{u\in \T_{k_n} }  \P_{k_n} ( W_{n-k_n} (u)+1 \geq \Pi_{k_n})\\ &=  W_{k_n} \Big[r_n \P( W_{n-k_n}+1 \geq r_n )\Big]_{r_n=\Pi_{k_n}}
   \\&\leq W_{k_n} \Big[\E\big((W_{n-k_n}+1 ) \mathbf{1}_{ \{W_{n-k_n}+1 \geq r_n\}} \big)\Big]_{r_n=\Pi_{k_n}}
  \\ &\leq W_{k_n} \Big[\E\big((W^*+1 ) \mathbf{1}_{ \{W^*+1 \geq r_n\}} \big)\Big]_{r_n=\Pi_{k_n}}
  \\   & \leq  W^*(\ln \Pi_{k_n})^{-\lambda}\E (W^*+1)(\ln (W^*+1))^{\lambda}
  \\ &\leq K_\xi W^* n^{-\lambda \beta} \E (W^*+1)(\ln (W^*+1))^{\lambda},
\end{align*}
where the last inequality holds since
\begin{equation}\label{cbrweq4.9}
\frac{1}{n}  \ln \Pi_{n} \rightarrow \E\ln m_0>0   \mbox{ a.s. },
\end{equation}
 and  $k_n\sim n^{\beta}$.
By the choice of $\beta$ and Lemma \ref{lem5}, we obtain  \eqref{cbrweq3-7}.

\medskip
\noindent {\bf Step 2}.  We next prove that $\forall \varepsilon>0$,
\begin{equation}\label{cbrweq3-8}
 \sum_{n=1}^{\infty}{\P}_{k_n} ( |\overline{A}_n -{\E}_{\xi,k_n} \overline{A}_n|>\frac{ \varepsilon}{\sqrt{n}}) <\infty.
\end{equation}

Take  a constant $b \in  (1, e^{\E\ln m_0})$. Observe  that $\forall u\in \T_{k_n}, n\geq 1$,
\begin{eqnarray*}
   \E_{k_n} \bar{X}^2_{n,u} &=&  \int_{0}^\infty 2x\P_{k_n} (|\bar{X}_{n,u}|>x  )  dx
    =  2\int_0^\infty x\P_{k_n}  ( |{X}_{n,u} | \mathbf{1 }_{ \{|{X}_{n,u} |<\Pi_{k_n}  \}} >x ) dx\\
   &\leq & 2\int_0^{\Pi_{k_n}} x\P_{k_n} ( | W_{n-k_n}(u)+1 | >x) dx= 2\int_0^{\Pi_{k_n}} x\P ( | W_{n-k_n}+1 | >x) dx \\
    &\leq&  2 \int_0^{\Pi_{k_n} }  x\P( W^*+1 >x) dx \\
    & \leq& 2   \int_e^{\Pi_{k_n} } (\ln x)^{-\lambda} \E (W^*+1)(\ln(W^*+1) )^{\lambda}  dx + 9\\
    &\leq & 2 \E (W^*+1)(\ln(W^*+1) )^{\lambda} \left(\int_e^{b^{k_n}} (\ln x)^{-\lambda} dx + \int_{b^{k_n}}^{\Pi_{k_n}} (\ln x)^{-\lambda} dx \right)+9\\
    & \leq & 2 \E (W^*+1)(\ln(W^*+1) )^{\lambda}(b^{k_n}  + (\Pi_{k_n}-b^{k_n}) (k_n\ln b  )^{-\lambda} )+9.
\end{eqnarray*}
Then  we have that
\begin{align*}
&~ ~~\sum_{n=1}^{\infty} {\P}_{k_n}(|\overline{A}_n-{\E}_{\xi,k_n} \overline{A}_n|>\frac{\varepsilon}{\sqrt{n}})\\  &=
 \sum_{n=1}^{\infty} {\E}_{k_n}{\P}_{\xi,k_n} (|\overline{A}_n-{\E}_{\xi,k_n} \overline{A}_n|>\frac{\varepsilon}{\sqrt{n}}) \\
    &\leq \varepsilon^{-2}\sum_{n=1}^{\infty} n{\E}_{k_n}\left( {\Pi_{k_n}^{-2}}    \sum_{u\in \T_{k_n}}{\E}_{\xi,k_n} \overline{X}_{n,u}^2\right)= \varepsilon^{-2} \sum_{n=1}^{\infty}n \left( {\Pi_{k_n}^{-2} }    \sum_{u\in \T_{k_n}}{\E}_{k_n} \overline{X}_{n,u}^2\right)\\
   & \leq \varepsilon^{-2} \sum_{n=1}^{\infty} \frac{nW_{k_n}}{\Pi_{k_n}}  \big [2 \E (W^*+1)(\ln(W^*+1)^{\lambda} )(b^{k_n}  + (\Pi_{k_n}-b^{k_n}) (k_n\ln b  )^{-\lambda} )+9\big ]\\
    & \leq  2\varepsilon^{-2}W^* \E (W^*+1)(\ln(W^*+1)^{\lambda} ) \bigg(  \sum_{n=1}^{\infty} \frac{n}{\Pi_{k_n}}b^{k_n} + \sum_{n=1}^{\infty} n (k_n\ln b  )^{-\lambda}   \bigg) +9\varepsilon^{-2}W^*  \sum_{n=1}^{\infty} \frac{n}{\Pi_{k_n}}.
\end{align*}
By \eqref{cbrweq4.9} and $\lambda \beta >2$, the three series in the last expression above converge under our hypothesis and hence \eqref{cbrweq3-8} is proved.

\medskip
\noindent {\bf Step 3.} Observe
\begin{eqnarray*}
 & & \P_{k_n} \Bigg(| \E_{\xi,k_n} \bar{A}_n | > \frac{\varepsilon}{\sqrt{n}} \Bigg )  \\
   & \leq &  \frac{\sqrt{n}}{\varepsilon}  \E_{k_n}  | \E_{\xi,k_n} \bar{A}_n |
   = \frac{\sqrt{n}}{\varepsilon} \E_{k_n} \Big|   \frac{1}{\Pi_{k_n} }   \sum_{u\in \T_{k_n}}  \E_{\xi,k_n} \bar{X}_{n,u} \Big|\\ & = &\frac{\sqrt{n}}{\varepsilon} \E_{k_n} \Big|   \frac{1}{\Pi_{k_n} }   \sum_{u\in \T_{k_n}}  (- \E_{\xi,k_n } X_{n,u} \mathbf{1}_{ \{ |X_{n,u}| \geq \Pi_{k_n}\}} ) \Big|
   \\ & \leq &  \frac{\sqrt{n}}{\varepsilon} \frac{1}{\Pi_{k_n} }   \sum_{u\in \T_{k_n}}  \E_{k_n } ( W_{n-k_n} (u) +1) \mathbf{1}_{ \{ W_{n-k_n}(u)+1 \geq \Pi_{k_n}\}} \\ & = &    \frac{\sqrt{n}W_{k_n}}{\varepsilon} \Big[ \E ( W_{n-k_n} +1) \mathbf{1}_{ \{ W_{n-k_n}+1 \geq r_n\}} \Big]_{r_n =\Pi_{k_n}}\\ &\leq& \frac{W^*}{\varepsilon}\sqrt{n} \Big[ \E ( W^* +1) \mathbf{1}_{ \{ W^*+1 \geq r_n\}} \Big]_{r_n =\Pi_{k_n}}\\
   &\leq &\frac{W^*}{\varepsilon}  \frac{\sqrt{n}}{(\ln\Pi_{k_n})^{\lambda} } \E    (W^*+1)  \ln ^{\lambda } ( W^*+1 )
   \\ & \leq & \frac{W^*}{\varepsilon}  K_\xi n^{\frac{1}{2} -\lambda\beta}  \E    (W^*+1)  \ln ^{\lambda } ( W^*+1 ).
 \end{eqnarray*}
Then  by \eqref{cbrweq4.9} and  $ \lambda \beta >2$,   it follows that \[ \sum_{n=1}^\infty \P_{k_n} \Bigg(| \E_{\xi,k_n} \bar{A}_n | > \frac{\varepsilon}{\sqrt{n}} \Bigg )<\infty. \]

Combining Steps 1-3, we obtain \eqref{eq10}. Hence  the lemma is proved.
\end{proof}

\begin{proof}[Proof of Lemma \ref{lem2} ]
For ease of notation, set \[ D_1(t)= (1-t^2) \phi(t), \qquad \kappa_{1,n}=  \frac{s_n^{(3)} -s_{k_n}^{(3)} }{6 (s_n^2-s_{k_n}^2 )^{3/2}}.\]

 Observe that
\begin{equation}\label{cbrweq3-16}
 \mathds{B}_n=\mathds{B}_{n1}+\mathds{B}_{n2}+\mathds{B}_{n3}+\mathds{B}_{n4},
\end{equation}
where
\begin{align*}
  \mathds{B}_{n1}  &= \frac{1}{\Pi_{k_n}} \sum_{u\in\T_{k_n}} \Bigg( \E_{\xi,k_n} W_{n-k_n} (u,s_n t-
S_u)- \Phi\bigg( \frac{s_n t-S_u}{( s_n^2-s_{k_n}^2)^{1/2}}\bigg)-\kappa_{1,n} D_1\bigg( \frac{s_n t-S_u}{( s_n^2-s_{k_n}^2)^{1/2}}\bigg)\Bigg);
\\   \mathds{B}_{n2}   &  = \frac{1}{\Pi_{k_n}} \sum_{u\in\T_{k_n}}\left( \Phi\bigg( \frac{s_n t-S_u}{( s_n^2-s_{k_n}^2)^{1/2}}\bigg)-\Phi(t) \right);\\
\mathds{B}_{n3}&=\kappa_{1,n}\frac{1}{\Pi_{k_n}} \sum_{u\in\T_{k_n}} \left( D_1\bigg( \frac{s_n t-S_u}{( s_n^2-s_{k_n}^2)^{1/2}}\bigg)-D_1(t) \right);\\
\mathds{B}_{n4}&=\kappa_{1,n} D_1(t) W_{k_n}.
\end{align*}
Then the lemma will be proved once  we  show that
\begin{align}
\label{cbrweq3-17}   & \sqrt{n}\mathds{B}_{n1} \xrightarrow{n\rightarrow\infty} 0; \\
 \label{cbrweq3-18}   & \sqrt{n}\mathds{B}_{n2} \xrightarrow{n\rightarrow\infty} -(\E \sigma_0^{(2)})^{-\frac{1}{2}}\phi(t)V_1 ; \\
 \label{cbrweq3-19}   & \sqrt{n}\mathds{B}_{n3} \xrightarrow{n\rightarrow\infty} 0; \\
\label{cbrweq3-20}    & \sqrt{n}\mathds{B}_{n4} \xrightarrow{n\rightarrow\infty} \frac{1}{6}{\E \sigma_0^{(3)}}{(\E \sigma_0^{(2)})^{-\frac{3}{2}} }D_1(t) W.
\end{align}
 We will prove these results subsequently.

We  first prove \eqref{cbrweq3-17}. The proof will  mainly be based on  the  following result about  asymptotic expansion of the distribution of the sum of independent random variables:
\begin{prop}  \label{prop4.5}
 Under the hypothesis of  Theorem \ref{th1},  for a.e. $\xi$,
\begin{equation*}
\epsilon_n=n^{1/2}\sup_{x\in \R}\Bigg|\P_{\xi } \bigg(\frac{\sum_{k=k_n}^{n-1}\widehat{L}_{k}}{( s_n^2-s_{k_n}^2)^{1/2}} \leq x \bigg)-  \Phi (x) -\kappa_{1,n}D_1(x) \Bigg| \xrightarrow{n\rightarrow\infty} 0.
\end{equation*}
\end{prop}
\begin{proof}
Let $X_k=0$  for $0\leq k \leq k_n-1$ and  $X_k= \widehat{L}_{k}$ for $k_n\leq k \leq n-1$. Then the random variables $\{X_k\}$  are independent under $\P_\xi$. Denote by $ v_k(\cdot)$ the characteristic function of $X_k$: $ v_k(t):= \E_\xi e^{it X_k} $. Using the Markov inequality and Lemma \ref{lem-Edge-exp}, we obtain the following result:
\begin{align*}
    &\sup_{x\in \R} \Bigg|\P_{\xi } \bigg(\frac{\sum_{k=k_n}^{n-1}\widehat{L}_{k}}{( s_n^2-s_{k_n}^2)^{1/2}} \leq x \bigg)-  \Phi (x) -\kappa_{1,n}D_1(x) \Bigg|    \\
\leq  & K_\xi  \left\{ (s_n^2-s_{k_n}^2 )^{-2} \sum_{j=k_n }^{n-1} \E_\xi | \widehat{L}_{j}|^4+  n^6 \left(\sup_{|t| >T } \frac{1}{n} \bigg( k_n+ \sum_{j=k_n}^{n-1} |v_j(t)| \bigg)+ \frac{1}{2n}\right)^n\right\}.
\end{align*}
By our conditions on the environment, we know that
\begin{equation}\label{cbrweq3.18}
\lim_{n\rightarrow \infty} n  {(s_n^2-s_{k_n}^2)^{-2}} \sum_{j=k_n}^{n-1} \E_\xi |\widehat{L}_k|^4  = \E |\widehat{L}_0|^4/ (\E \sigma_0^{(2)})^2.
\end{equation}
By  \eqref{cbrweq2-3}, $\widehat{L}_n $
satisfies
\begin{equation*}
 \P\Big(  \limsup_{|t|\rightarrow\infty}|v_n(t)|<1 \Big) >0.
\end{equation*}
So there exists a  constant $c_n \leq 1$ depending on $\xi_n$  such that
\begin{equation*}
  \sup_{|t|> T} |v_n(t)| \leq c_n\quad   \mbox{ and }  \quad  \P(c_n <1) >0.
 \end{equation*}
Then $\E c_0 <1$. By the Birkhoff ergodic theorem, we have
\begin{align*}
                \sup_{|t|> T}    \bigg(  \frac{1}{n}\sum_{j=k_n}^{n-1} |v_j(t)|\bigg) &\leq  \frac{1}{n}\sum_{j=1}^{n-1} c_j \rightarrow \E c_0<1.
                                 \end{align*}
Then for $n$ large enough,
\begin{equation}\label{cbrweq3-22a}
  \left(\sup_{|t| >T } \frac{1}{n} \bigg( k_n+ \sum_{j=k_n}^{n-1}  |v_j(t)| \bigg)+ \frac{1}{2n}\right)^n=o( n^{-m}), \quad  \forall  m >0.
\end{equation}
From   \eqref{cbrweq3.18} and \eqref{cbrweq3-22a}, we get the conclusion of the proposition.
\end{proof}
From Proposition \ref{prop4.5}, it is easy to see that
\begin{equation*}
 \sqrt{n}  |\mathds{B}_{n1} | \leq  W_{k_n} \epsilon_n \xrightarrow{n\rightarrow\infty} 0.
\end{equation*}
Hence \eqref{cbrweq3-17} is proved.

We next prove \eqref{cbrweq3-18}.
Observe that
     \begin{align*}
      &\mathds{B}_{n2}   = \mathds{B}_{n21}+ \mathds{B}_{n22}+ \mathds{B}_{n23} + \mathds{B}_{n24}+\mathds{B}_{n25},    \\
       \mbox{with}~~ &\mathds{B}_{n21} = \frac{1}{\Pi_{k_n}} \sum_{u\in \T_{k_n}}   \left[\Phi \bigg( \frac{s_n t-S_u}{( s_n^2-s_{k_n}^2)^{1/2}}\bigg)  -\Phi(t) - \phi(t) \bigg(  \frac{s_n t-S_u}{( s_n^2-s_{k_n}^2)^{1/2}}- t\bigg)\right]\mathbf{1}_{\{|S_u|\leq k_n\}},\\
       &  \mathds{B}_{n22} = \frac{1}{\Pi_{k_n}} \sum_{u\in \T_{k_n}}   \left[\Phi \bigg( \frac{s_n t-S_u}{( s_n^2-s_{k_n}^2)^{1/2}}\bigg)  -\Phi(t) \right]\mathbf{1}_{\{|S_u|>k_n\}},\\
       & \mathds{B}_{n23}=-  \frac{1}{\Pi_{k_n}} \sum_{u\in \T_{k_n}}\bigg(  \frac{s_n t-S_u}{( s_n^2-s_{k_n}^2)^{1/2}}- t\bigg)\phi(t)\mathbf{1}_{\{|S_u|>k_n\}},\\
       & \mathds{B}_{n24}= \frac{1}{( s_n^2-s_{k_n}^2)^{1/2}}\big(s_n- ( s_n^2-s_{k_n}^2)^{1/2}\big) W_{k_n} \phi(t)t,\\
       & \mathds{B}_{n25}= -\frac{1}{( s_n^2-s_{k_n}^2)^{1/2}} \phi(t) N_{1,k_n}.
     \end{align*}
By  Taylor's formula and the choice of $\beta$ and $k_n$, we get
\begin{align*}
  \widetilde{\epsilon}_n=  & \sqrt{n}\sup_{|y|\leq k_n} \left| \Phi \bigg( \frac{s_n t-y}{( s_n^2-s_{k_n}^2)^{1/2}}\bigg)  -\Phi(t) - \phi(t) \bigg(  \frac{s_n t-y}{( s_n^2-s_{k_n}^2)^{1/2}}- t\bigg) \right| \\
    &  \leq \sqrt{n }\sup_{|y|  \leq k_n}\bigg| \frac{s_n t-y}{( s_n^2-s_{k_n}^2)^{1/2}}- t\bigg| ^2 \xrightarrow{n\rightarrow\infty} 0.
\end{align*}
Thus
\begin{equation}\label{cbrweq3.20}
  |\sqrt{n} \mathds{B}_{n21}| \leq   W_{k_n} \widetilde{\epsilon}_n \xrightarrow{n\rightarrow\infty} 0.
\end{equation}

We continue to prove that
\begin{equation}\label{cbrweq3.21}
 \sqrt{ n}\mathds{B}_{n22} \xrightarrow{n\rightarrow\infty} 0;  \qquad \sqrt{n}\mathds{B}_{n23}\xrightarrow{n\rightarrow\infty} 0.
\end{equation}
This will follow from the facts:
  \begin{equation}\label{cbrweq3.22}
       \frac{1}{\Pi_{k_n}} \sum_{u\in \T_{k_n}} |S_u|\mathbf{1}_{\{|S_u|>k_n\}} \xrightarrow{n\rightarrow \infty} 0  \mbox{~ a.s.};~~    \sqrt{n}  \frac{1}{\Pi_{k_n}} \sum_{u\in \T_{k_n}} \mathbf{1}_{\{S_u|>k_n\}} \xrightarrow{n\rightarrow \infty} 0 ~ \mbox{a.s.}
  \end{equation}
In order to prove \eqref{cbrweq3.22}, we firstly observe that
\begin{align*}
&\E \left( \sum_{n=1}^\infty \frac{1}{\Pi_{k_n}} \sum_{u\in \T_{k_n}} |S_u|\ind{|S_u|>k_n} \right)\\ = \ &\sum_{n=1}^\infty   \E |\widehat{S}_{k_n}| \ind{|\widehat{S}_{k_n}|>k_n }  \leq\ \sum_{n=1}^\infty  k_n^{1-\eta} \E|\widehat{S}_{k_n}|^{\eta}
   \leq\sum_{n=1}^\infty  k_n^{-\frac{\eta}{2}}  \sum_{j=0}^{k_n-1} \E |\widehat{L}_j|^{\eta}
           =\sum_{n=1}^\infty  k_n^{1-\frac{\eta}{2}}\E |\widehat{L}_0|^{\eta}, \\&
   \E \left( \sum_{n=1}^\infty \sqrt{n} \frac{1}{\Pi_{k_n}} \sum_{u\in \T_{k_n}} \ind{|S_u|>k_n} \right)
   \\ = \ &\sum_{n=1}^\infty  \sqrt{n} \E  \ind{|\widehat{S}_{k_n}|>k_n }  \leq\ \sum_{n=1}^\infty  \sqrt{n} k_n^{-\eta} \E|\widehat{S}_{k_n}|^{\eta}
   \leq\sum_{n=1}^\infty  \sqrt{n} k_n^{-\frac{\eta}{2}-1}  \sum_{j=0}^{k_n-1} \E |\widehat{L}_j|^{\eta}
           =\sum_{n=1}^\infty n^{\frac{1}{2}} k_n^{-\frac{\eta}{2}}\E |\widehat{L}_0|^{\eta}.
\end{align*}
The assumptions  on $\beta$,  $k_n$  and $\eta$ ensure that the series in the right hand side of the  above two expressions converge.
Hence $$ \sum_{n=1}^\infty \frac{1}{\Pi_{k_n}} \sum_{u\in \T_{k_n}} |S_u|\ind{|S_u|>k_n} <\infty,  \quad  \sum_{n=1}^\infty \sqrt{n} \frac{1}{\Pi_{k_n}} \sum_{u\in \T_{k_n}} \ind{|S_u|>k_n} <\infty\mbox{ ~  a.s.,} $$
which deduce  \eqref{cbrweq3.22},  and consequently,   \eqref{cbrweq3.21} is proved.

By the Birkhoff ergodic theorem,  we have
\begin{equation}\label{cbrweq3.23}
   \lim_{n\rightarrow \infty} \frac{s_n^2}{n} = \E \sigma_0^{(2)},
\end{equation}
whence   by the choice of $\beta<1/4$ and the conditions  on the environment,
\begin{equation}\label{cbrweq3.24}
   \sqrt{n} \mathds{B}_{n24}=  \frac{\sqrt{n}}{( s_n^2-s_{k_n}^2)^{1/2}}\frac{s_{k_n}^2}{s_n+ ( s_n^2-s_{k_n}^2)^{1/2}} W_{k_n} \phi(t)t\xrightarrow{n\rightarrow\infty} 0.
\end{equation}

Due to  Proposition \ref{th1a} and \eqref{cbrweq3.23},  we conclude  that
\begin{equation}\label{cbrweq3.25}
 \sqrt{n}\mathds{B}_{n25} \xrightarrow{n\rightarrow \infty}-(\E \sigma_0^{(2)})^{-\frac{1}{2}}\phi(t)V_1~~~ \mbox{a.s.}
\end{equation}

From \eqref{cbrweq3.20}, \eqref{cbrweq3.21}, \eqref{cbrweq3.24} and \eqref{cbrweq3.25}, we derive \eqref{cbrweq3-18}.

Now we turn to the proof of \eqref{cbrweq3-19}.

According to the hypothesis of Theorem \ref{th1},  it follows from the Birkhoff ergodic theorem that
\begin{equation}\label{cbrweq3-26}
   \lim_{n\rightarrow \infty } \sqrt{n}\kappa_{1,n}= \frac{1}{6} (\E \sigma_0^{(2)})^{-3/2} \E \sigma_0^{(3)}.
\end{equation}
Notice that
   \begin{eqnarray*}
       && \left|\frac{1}{\Pi_{k_n}} \sum_{u\in\T_{k_n}} \left( D_1\bigg( \frac{s_n t-S_u}{( s_n^2-s_{k_n}^2)^{1/2}}\bigg)-D_1(t) \right)\right|   \\
       &\leq &  \frac{2}{\Pi_{k_n}} \sum_{u\in\T_{k_n}} \mathbf{1}_{\{|S_u|>k_n\}} +   \frac{1}{\Pi_{k_n}} \sum_{u\in\T_{k_n}} \left| D_1\bigg( \frac{s_n t-S_u}{( s_n^2-s_{k_n}^2)^{1/2}}\bigg)-D_1(t) \right|\mathbf{1}_{\{|S_u|\leq k_n\}}.
   \end{eqnarray*}
The first term in the last expression above tends to 0  a.s. by \eqref{cbrweq3.22}, and the second one tends to 0 a.s. because the martingale $\{W_n\}$ converges and
\begin{equation*}
  \sup_{|y|\leq k_n} \left|D_1\Bigg( \frac{s_n t-y}{( s_n^2-s_{k_n}^2)^{1/2}}\Bigg)-D_1(t)\right | \xrightarrow{n\rightarrow \infty} 0 .
\end{equation*}
Combining the above results, we obtain \eqref{cbrweq3-19}.

It remains to prove \eqref{cbrweq3-20}, which is  immediate from  \eqref{cbrweq3-26} and the fact $W_n\xrightarrow{n\rightarrow\infty} W$.

So   Lemma \ref{lem2} has been proved.
\end{proof}
\begin{proof}[Proof of Lemma \ref{lem3}]
This lemma  follows  from  the following result given in \cite{HuangLiu}.
\begin{prop}[ \cite{HuangLiu} ]\label{pro3}
Assume the condition  \eqref{cbrweq1}.
 Then
\begin{equation*}
W-W_n=o(n^{-\lambda})\qquad a.s.
\end{equation*}
\end{prop}
By the choice of $\beta$ and $ k_n$, we see that
\begin{equation*}
   \sqrt{n}(W-W_{k_n}) =o(n^{\frac{1}{2} -\lambda\beta}) \xrightarrow{n\rightarrow\infty} 0.
\end{equation*}
\end{proof}
Now  Theorem \ref{th1} follows from the decomposition \eqref{cbrweq12} and Lemmas \ref{lem1} -- \ref{lem3}.
\section{Proof of Theorem \ref{th2}}\label{sec4}

We will  follow the similar procedure as in the proof of Theorem \ref{th1}.


%

We remind that $\lambda,\eta>16 $  in the current setting.  Hereafter we will choose  $\max \{ \frac{4}{\lambda}, \frac{4}{\eta}\}<\beta <\frac{1}{4} $   and let $k_n= \lfloor n^{\beta}\rfloor $ (the integral part of $n^{\beta}$).

By \eqref{cbrweq11}, we have
\begin{equation}\label{cbrweq4-1}
 \sqrt{2\pi} s_n \Pi_n^{-1} Z_n(A) -W \int_A \exp \{ -\frac{x^2}{2s_n^2}\} dx =  \Lambda_{1,n} +\Lambda_{2,n }+\Lambda_{3,n} ,
\end{equation}
\begin{eqnarray*}
 \mbox{~~ with ~~ } \Lambda_{1,n}  &=&   \sqrt{2\pi} s_n \Pi_{k_n}^{-1} \sum_{u\in \T_{k_n}} \bigg( W_{n-k_n} (u, A-S_u) -  \E_{\xi,k_n} W_{n-k_n} (u, A-S_u) \bigg);   \\
  \Lambda_{2,n}  &=&  \Pi_{k_n}^{-1} \sum_{u\in \T_{k_n}} \bigg ( \sqrt{2\pi} s_n  \E_{\xi,k_n} W_{n-k_n} (u, A-S_u)  - \int_A \exp \{ -\frac{x^2}{2s_n^2}\} dx \bigg ); \\   \Lambda_{3,n}  &=& ( W_{k_n} -W)  \int_A \exp \{ -\frac{x^2}{2s_n^2}\} dx .
\end{eqnarray*}
On basis of this decomposition, we shall divide the proof of Theorem \ref{th2} into the following lemmas.
\begin{lem}\label{lem4-1}
Under the hypothesis of  Theorem \ref{th2},  a.s.
\begin{equation}\label{eq4-6}
 {n}\Lambda_{1,n} \xrightarrow{n \rightarrow \infty } 0.
\end{equation}
\end{lem}
\begin{lem}\label{lem4-2} Under the hypothesis of  Theorem \ref{th2},  a.s.
\begin{multline}\label{eq4-7}
 {n} \Lambda_{2,n}  \xrightarrow{n \rightarrow \infty }  ( \E \sigma_0^{(2)})^{-1}(\frac{1}{2} V_2 +\overline{x}_A V_1) |A|+ \frac{1}{2}{\E \sigma_0^{(3)} } ( \E \sigma_0^{(2)})^{-2}(V_1-\overline{x}_A  W )|A| \\
~+\frac{1}{8}  (\E \sigma_0^{(2)} )^{-2}\E(\sigma_0^{(4)}-3(\sigma_0^{(2)} )^2 ) W |A| - \frac{5}{24}( \E \sigma_0^{(2)} )^{-3}(\E\sigma_0^{(3)})^2 W |A| .
\end{multline}
\end{lem}
\begin{lem}\label{lem4-3} Under the hypothesis of  Theorem \ref{th2},  a.s.
\begin{equation}\label{eq4-8}
{n}\Lambda_{3,n} \xrightarrow{n \rightarrow \infty } 0.
\end{equation}
\end{lem}

Now we go to prove the lemmas subsequently.

\begin{proof}[Proof of Lemma \ref{lem4-1}]
The proof of   Lemma \ref{lem4-1}  follows the same procedure as that of  Lemma \ref{lem1} with minor changes in scaling. We omit the details.
\end{proof}

\begin{proof}[Proof of Lemma \ref{lem4-2}]
We start the proof by introducing  some notation: set
\begin{align*}
  \kappa_{1, n}   =  &{\frac{1}{6} (s_n^2-s_{k_n}^2 )^{-3/2} }{(s_n^{(3)} - s_{k_n}^{(3)}) }, \quad \kappa_{2, n}  =   \frac{1}{72} (s_n^2-s_{k_n}^2 )^{-3} (s_n^{(3)} - s_{k_n}^{(3)})^2, \\
  \kappa_{3, n}  = &\frac{ 1}{24} (s_n^2-s_{k_n}^2 )^{-2} \sum_{j=k_n}^{n-1} \Big(\sigma_j^{(4)} -3\big(\sigma_j^{(2)}\big)^2 \Big).
  \end{align*}
  Define for $x\in \R$,
  \begin{align*}
 D_1(x)= &-H_2(x)\phi(x) ,  ~  D_2(x)= -H_5(x)\phi(x), ~
  D_3(x)= -H_3(x) \phi(x),\\
 R_n(x) = &-\frac{\Big(s_n^{(3)}-s_{k_n}^{(3)} \Big)^3}{1296 (s_n^2-s_{k_n}^2)^{9/2}} H_8(x) \phi(x) -\frac{\sum_{j=k_n}^{n-1} \Big(\sigma_j^{(5)} -10\sigma_j^{(3)}\sigma_j^{(2)} \Big)}{120 (s_n^2-s_{k_n}^2)^{5/2}}H_4(x)  \phi(x)\\ & -\frac{ \Big(s_n^{(3)}-s_{k_n}^{(3)}\Big ) \sum_{j=k_n}^{n-1} \Big(\sigma_j^{(4)} -3\big(\sigma_j^{(2)}\big)^2  \Big)} { 144  (s_n^2-s_{k_n}^2)^{7/2}}H_6(x)  \phi(x),
\end{align*}
where $H_m$ are  Chebyshev-Hermite polynomials defined in \eqref{eqCH}.
 We decompose $\Lambda_{2,n}$ into 7 terms:
 \begin{align}\label{eq4-14}
  \Lambda_{2,n}=&\Lambda_{2,n1} +  \Lambda_{2,n2}+\Lambda_{2,n3}+\Lambda_{2,n4} +\Lambda_{2,n5}+\Lambda_{2,n6} + \Lambda_{2,n7},
  \end{align}
where
   \begin{align*}
  \Lambda_{2,n1}=&~\sqrt{2\pi} s_n \Pi_{k_n}^{-1} \sum_{u\in \T_{k_n}}   \Bigg[ \E_{\xi,k_n} W_{n-k_n} (u, A-S_u) - \int_A\Bigg(  \phi\bigg( \frac{x-S_u}{({s_n^2-s_{k_n}^2})^{1/2}}\bigg) \\
 \nonumber    & ~+ \sum_{\nu=1}^3\kappa_{\nu,n} D'_\nu \bigg( \frac{x-S_u}{({s_n^2-s_{k_n}^2})^{1/2}}\bigg) +R'_n\bigg( \frac{x-S_u}{({s_n^2-s_{k_n}^2})^{1/2}}\bigg) \Bigg)  \frac{ dx}{({s_n^2-s_{k_n}^2})^{1/2}} \Bigg] , \\
 \nonumber \Lambda_{2,n2} =&  ~ \Pi_{k_n}^{-1} \sum_{u\in \T_{k_n}} \mathbf{1}_{\{|S_u|\leq k_n\}} \int_A \bigg[  \frac{s_n}{({s_n^2-s_{k_n}^2})^{1/2}} \exp\{-\frac{(x-S_u)^2}{2({s_n^2-s_{k_n}^2})} \}- \exp\{-\frac{x^2}{2s_n^2}\} \bigg] dx , \\
 \nonumber\Lambda_{2,n3} =&~\frac{ \sqrt{2\pi}\kappa_{1,n}s_n}{({s_n^2-s_{k_n}^2})^{1/2}} \Pi_{k_n}^{-1} \sum_{u\in \T_{k_n}} \mathbf{1}_{\{|S_u|\leq k_n\}} \int_A D'_1 \bigg( \frac{x-S_u}{({s_n^2-s_{k_n}^2})^{1/2}}\bigg) dx,\\
  \nonumber\Lambda_{2,n4} =& ~\frac{ \sqrt{2\pi}\kappa_{2,n}s_n}{({s_n^2-s_{k_n}^2})^{1/2}} \Pi_{k_n}^{-1} \sum_{u\in \T_{k_n}}\mathbf{1}_{\{|S_u|\leq k_n\}} \int_A D'_2 \bigg( \frac{x-S_u}{({s_n^2-s_{k_n}^2})^{1/2}}\bigg) dx,\\
 \nonumber \Lambda_{2,n5} =& ~ \frac{ \sqrt{2\pi}\kappa_{3,n}s_n}{({s_n^2-s_{k_n}^2})^{1/2}} \Pi_{k_n}^{-1} \sum_{u\in \T_{k_n}} \mathbf{1}_{\{|S_u|\leq k_n\}} \int_A D'_3 \bigg( \frac{x-S_u}{({s_n^2-s_{k_n}^2})^{1/2}}\bigg) dx,\\
 \nonumber \Lambda_{2,n6} =& ~ \frac{ \sqrt{2\pi}s_n}{({s_n^2-s_{k_n}^2})^{1/2}}\Pi_{k_n}^{-1} \sum_{u\in \T_{k_n}} \mathbf{1}_{\{|S_u|\leq k_n\}}  \int_A R'_n \bigg( \frac{x-S_u}{({s_n^2-s_{k_n}^2})^{1/2}}\bigg) dx, \\
 \nonumber\Lambda_{2,n7}= &~ ~ \frac{ \sqrt{2\pi}s_n}{({s_n^2-s_{k_n}^2})^{1/2}} \Pi_{k_n}^{-1} \sum_{u\in \T_{k_n}}
 \Bigg ( \int_A \bigg( \phi \bigg( \frac{x-S_u}{({s_n^2-s_{k_n}^2})^{1/2}}\bigg)+R_n\bigg( \frac{x-S_u}{({s_n^2-s_{k_n}^2})^{1/2}}\bigg) \\
  \nonumber  &~~+\sum_{\nu=1}^3\kappa_{\nu,n} D'_\nu \bigg( \frac{x-S_u}{({s_n^2-s_{k_n}^2})^{1/2}}\bigg) - \Big({1- \frac{s^2_{k_n}}{s^2_n}}\Big)^{1/2}  \phi(x/s_n)  \bigg)dx \Bigg )\mathbf{1}_{\{|S_u|>k_n\}} .
 \end{align*}
The lemma will follow once we prove that a.s.
\begin{align}
\label{eq4-15}
    & n \Lambda_{2,n1} \xrightarrow{n \rightarrow \infty } 0,   \\
\label{eq4-16}  &  n \Lambda_{2,n2} \xrightarrow{n \rightarrow \infty } ( \E \sigma_0^{(2)})^{-1}(\frac{1}{2} V_2 +\overline{x}_A V_1) |A|,    \\
\label{eq4-17} &  n \Lambda_{2,n3} \xrightarrow{n \rightarrow \infty } \frac{1}{2}{\E \sigma_0^{(3)} } ( \E \sigma_0^{(2)})^{-2}(V_1-\overline{x}_A  W )|A|,  \\
\label{eq4-18} &  n \Lambda_{2,n4} \xrightarrow{n \rightarrow \infty } - \frac{5}{24}( \E \sigma_0^{(2)} )^{-3}(\E\sigma_0^{(3)})^2 W |A|, \\
\label{eq4-19} &  n \Lambda_{2,n5} \xrightarrow{n \rightarrow \infty }  \frac{1}{8}  (\E \sigma_0^{(2)} )^{-2}\E(\sigma_0^{(4)}-3(\sigma_0^{(2)})^2  ) W |A|,\\
\label{eq4-20}
    & n \Lambda_{2,n6} \xrightarrow{n \rightarrow \infty } 0,\\
    \label{eq4-13}
    & n \Lambda_{2,n7} \xrightarrow{n \rightarrow \infty } 0.
\end{align}

The proof of \eqref{eq4-15} is based on the  following result on the asymptotic expansion of the distribution of the sum of independent random variables:
\begin{prop}\label{pro4} Under the hypothesis of  Theorem \ref{th2},  for a.e. $\xi$,
\begin{equation*}
\epsilon_n=n^{3/2}\sup_{x\in \R}\Bigg|\P_{\xi} \Bigg(\frac{\sum_{k=k_n}^{n-1}\widehat{L}_{k}}{( s_n^2-s_{k_n}^2)^{1/2}} \leq x \Bigg)-  \Phi (x) - \sum_{\nu=1}^{3} \kappa_{\nu,n}D_\nu(x) -R_n(x) \Bigg| \xrightarrow{n\rightarrow\infty} 0.
\end{equation*}
\end{prop}
\begin{proof}
   Let $X_k=0$  for $0\leq k \leq k_n-1$ and  $X_k= \widehat{L}_{k}$ for $k_n\leq k \leq n-1$. Then the random variables $\{X_k\}$  are independent under $P_\xi$. By Markov's inequality and  Lemma \ref{lem-Edge-exp} we obtain the following result:
\begin{align*}
    &\sup_{x\in \R} \Bigg|\P_{\xi} \bigg(\frac{\sum_{k=k_n}^{n-1}\widehat{L}_{k}}{{( s_n^2-s_{k_n}^2)^{1/2}}} \leq x \bigg)-  \Phi (x) - \sum_{\nu=1}^{3} \kappa_{\nu,n}D_\nu(x) -R_n(x)\Bigg|    \\
\leq  & K_\xi  \left\{ (s_n^2-s_{k_n}^2 )^{-3} \sum_{j=k_n }^{n-1} \E_\xi | L_{j}|^6+  n^{15} \left(\sup_{|t| >T } \frac{1}{n} \left( k_n+ \sum_{j=k_n}^{n-1} |v_j(t)| \right)+ \frac{1}{2n}\right)^n\right\}.
\end{align*}
By our conditions on the environment, we know that
\begin{equation}\label{cbrweq3-21}
\lim_{n\rightarrow \infty} n^{2}  {(s_n^2-s_{k_n}^2)^{-3}} \sum_{j=k_n}^{n-1} \E_\xi |\widehat{L}_k|^6  = \E |\widehat{L}_0|^6/ (\E \sigma_0^{(2)})^{3}.
\end{equation}
The required proposition  concludes from \eqref{cbrweq3-21} and \eqref{cbrweq3-22a}.
\end{proof}
Using Proposition \ref{pro4}, we deduce  that
\begin{align*}
  |n\Lambda_{2,n1}|  & \leq  {\sqrt{2\pi} s_n}{n^{-\frac{1}{2}}} W_{k_n} \epsilon_n \xrightarrow{n \rightarrow \infty } 0,
\end{align*}
and  \eqref{eq4-15} is proved.

Next we turn to the proof of \eqref{eq4-16}.
Using Taylor's expansion and the boundedness of  the set $A$, together with  the choice of $ \beta$ and $k_n$, we get that
\begin{equation*}
     \frac{s_n}{({s_n^2-s_{k_n}^2})^{1/2}} \exp\{-\frac{(x-y)^2}{2({s_n^2-s_{k_n}^2})} \}- \exp\{-\frac{x^2}{2s_n^2}\}= \frac{1}{2(s_n^2-s_{k_n}^2) } \big(s_{k_n}^2- y^2 +2 xy  + o(1)\big ),
\end{equation*}
uniformly for all $|y| \leq k_n$  and  $x\in A $ as $n\rightarrow \infty$.
By the same arguments as in the  proof of  \eqref{cbrweq3.22}, we can show that for $\eta>16$, with  $ \beta, k_n$ chosen above,
\begin{equation}\label{cbrweq4.18}
  n \Pi_{k_n}^{-1} \sum_{u\in \T_{k_n}} \mathbf{1}_{\{|S_u|>k_n\}} \xrightarrow{n\rightarrow \infty} 0 \quad \mbox{ and }\qquad \Pi_{k_n}^{-1} \sum_{u\in \T_{k_n}}S_u^2 \mathbf{1}_{\{|S_u|>k_n\}} \xrightarrow{n\rightarrow \infty} 0 \quad \mbox{ a.s.}
\end{equation}
Therefore  as $n$  tends to infinity, we have a.s.
 \begin{align*}
            n\Lambda_{2,n2}  = ~&  n \frac{1}{2(s_n^2-s_{k_n}^2) } \bigg(|A| \Pi_{k_n}^{-1} \sum_{u\in \T_{k_n}} (s_{k_n}^2- S_u^2 ) \mathbf{1}_{\{|S_u|\leq k_n\}} \\&  + 2  \int_A x dx\Pi_{k_n}^{-1} \sum_{u\in \T_{k_n}} S_u \mathbf{1}_{\{|S_u|\leq k_n\}}  +o(1)  \bigg)   \\
              = ~ &   \frac{n}{2(s_n^2-s_{k_n}^2) } \big( N_{2,k_n}|A| + 2|A|\overline{x}_A N_{1,k_n} +o(1) \big) \\
              =  ~ & (2\E \sigma_0^{(2)})^{-1}  (V_2  +2 \overline{x}_A  V_1) |A|  +o(1),
          \end{align*}
which proves \eqref{eq4-16}.

To prove \eqref{eq4-17}, we observe that
\begin{align*}
     & ~\Lambda_{2,n3}=  \frac{ \kappa_{1,n}s_n}{({s_n^2-s_{k_n}^2})^{1/2}} \Pi_{k_n}^{-1} \sum_{u\in \T_{k_n}} \mathbf{1}_{\{|S_u|\leq k_n\}} \int_A \bigg( \frac{(x-S_u)^{3}}{({s_n^2-s_{k_n}^2})^{3/2}}- \frac{3 (x-S_u)}{({s_n^2-s_{k_n}^2})^{1/2}} \bigg) e^{-\frac{(x-S_u)^2}{2({s_n^2-s_{k_n}^2})}}   dx    \\
  & ~ ~~~~~~=\Lambda_{2,n31}+\Lambda_{2,n32}+\Lambda_{2,n33}+\Lambda_{2,n34},
  \end{align*}
{with}  \begin{align*} & \Lambda_{2,n31}=~ \frac{ \kappa_{1,n}s_n}{({s_n^2-s_{k_n}^2})^{1/2}} \Pi_{k_n}^{-1} \sum_{u\in \T_{k_n}}  \mathbf{1}_{\{|S_u|\leq k_n\}} \int_A  \frac{(x-S_u)^{3}}{({s_n^2-s_{k_n}^2})^{3/2}}   e^{-\frac{(x-S_u)^2}{2({s_n^2-s_{k_n}^2})}}   dx ;   \\
&\Lambda_{2,n32}=~ \frac{ \kappa_{1,n}s_n}{({s_n^2-s_{k_n}^2})^{1/2}} \Pi_{k_n}^{-1} \sum_{u\in \T_{k_n}} \mathbf{1}_{\{|S_u|\leq k_n\}}  \int_A \frac{3 (x-S_u)}{({s_n^2-s_{k_n}^2})^{1/2}} \bigg (1- e^{-\frac{(x-S_u)^2}{2({s_n^2-s_{k_n}^2})}} \bigg)  dx ;\\
&\Lambda_{2,n33}=~- \frac{ \kappa_{1,n}s_n}{({s_n^2-s_{k_n}^2})^{1/2}} \Pi_{k_n}^{-1} \sum_{u\in \T_{k_n}} \int_A \frac{3 (x-S_u)}{({s_n^2-s_{k_n}^2})^{1/2}} dx;\\
&\Lambda_{2,n34}=~ \frac{ \kappa_{1,n}s_n}{({s_n^2-s_{k_n}^2})^{1/2}} \Pi_{k_n}^{-1} \sum_{u\in \T_{k_n}}  \mathbf{1}_{\{|S_u|> k_n\}} \int_A \frac{3 (x-S_u)}{({s_n^2-s_{k_n}^2})^{1/2}}   dx.
\end{align*}
It is clear  that \begin{align*}
                               & n|\Lambda_{2,n31}|\leq   \frac{n \kappa_{1,n}s_n}{({s_n^2-s_{k_n}^2})^{2}} \int_A (|x|+ k_n)^3 dx W_{k_n} \xrightarrow{n\rightarrow \infty} 0 \mbox{ a.s.},  \\
                               &  n|\Lambda_{2,n32}|\leq \frac{n \kappa_{1,n}s_n}{({s_n^2-s_{k_n}^2})^{2}} \int_A \frac{3}{2}(|x|+ k_n)^3 dx W_{k_n} \xrightarrow{n\rightarrow \infty} 0 \mbox{ a.s.  }  \quad  (  1- e^{-x} \leq x,  \mbox{ for } x>0  ),\\
                               &  n\/ \Lambda_{2,n33}~=   \frac{n (s_n^{(3)}- s_{k_n}^{(3)})s_n}{6({s_n^2-s_{k_n}^2})^{5/2}}\cdot 3|A|(N_{1,k_n}- \overline{x}_A W_{k_n}) \\
                                 &~~~~~~~~~~~\xrightarrow{n\rightarrow \infty}   \frac{1}{2}{\E \sigma_0^{(3)} } ( \E \sigma_0)^{-2}(V_1-\overline{x}_A  W )|A|  \mbox{ a.s.}, \\
                               &  n|\Lambda_{2,n34}|\leq \frac{3n \kappa_{1,n}s_n}{({s_n^2-s_{k_n}^2})} \bigg( \int_A|x|dx \Pi_{k_n}^{-1} \sum_{u\in \T_{k_n}}  \mathbf{1}_{\{|S_u|> k_n\}} + |A|\Pi_{k_n}^{-1} \sum_{u\in \T_{k_n}}  |S_u|\mathbf{1}_{\{|S_u|> k_n\}} \bigg) \\
                               & ~~~~~~~~~~~\xrightarrow{n\rightarrow \infty} 0 \mbox{ a.s.  }   (\mbox{ by }  \eqref{cbrweq3.22}),
                            \end{align*}
whence \eqref{eq4-17} follows.

By the Birkhoff ergodic theorem, we see that
\begin{equation}\label{eq4-21}
   \lim_{n\rightarrow \infty} \frac{ n \kappa_{2,n}s_n}{({s_n^2-s_{k_n}^2})^{1/2}} = \frac{(\E \sigma_0^{(3)})^2}{72(\E \sigma_0^{(2)})^{3} },  \quad \lim_{n\rightarrow \infty} \frac{ n\kappa_{3,n}s_n}{({s_n^2-s_{k_n}^2})^{1/2}} = \frac{\E(\sigma_0^{(3)} -3(\sigma_0^{(2)})^2 )}{24(\E \sigma_0^{(2)})^{2} }.
\end{equation}
Elementary calculus shows that,  uniformly for $|y|\leq k_n$
   \begin{align}\label{eq4-22}
    \mbox{ if }     \nu\geq 1,    \quad          &  \int_A  \bigg( \frac{x-y}{({s_n^2-s_{k_n}^2})^{1/2}}\bigg)^\nu \exp\bigg(-\frac{(x-y)^2}{2({s_n^2-s_{k_n}^2})}\bigg)dx  \xrightarrow{n\rightarrow \infty} 0 \mbox{ a.s.  }, \\
  \label{eq4-23} \mbox{ and }  ~~~~~~ &  \int_A  \exp\bigg(-\frac{(x-y)^2}{2({s_n^2-s_{k_n}^2})}\bigg)dx   \xrightarrow{n\rightarrow \infty} 1 \mbox{ a.s.  }
                     \end{align}
Combining \eqref{cbrweq4.18},\eqref{eq4-21}, \eqref{eq4-22} and \eqref{eq4-23},  we deduce  \eqref{eq4-18}   and  \eqref{eq4-19}.

By the Birkhoff ergodic theorem and the definition of $H_m(x)$ and $\phi(x)$, we see that
\begin{equation*}
\sup_{x\in \R} |R_n'(x) |=  O(\frac{1}{n^{3/2}}),
\end{equation*}
whence \eqref{eq4-20} follows.

Finally because  $|\Lambda_{2,n7}| $ is bounded by $K_\xi\cdot\Pi_{k_n}^{-1} \sum_{u\in \T_{k_n}} \mathbf{1}_{\{|S_u|>k_n\}}$,
\eqref{cbrweq4.18} implies \eqref{eq4-13}.

So the required result \eqref{eq4-7} follows from \eqref{eq4-15} -- \eqref{eq4-13}.
\end{proof}

\begin{proof}[Proof of Lemma \ref{lem4-3}]
By  Proposition \ref{pro3} ,  under our assumption, we have
\begin{equation*}
W-W_n=o(n^{-\lambda})\qquad a.s.
\end{equation*}
By the choice of $\beta$ and $ k_n$, we see that
\begin{equation*}
  n^{ \frac{3}{2}}(W-W_{k_n}) =o(n^{\frac{3}{2} -\lambda\beta}) \xrightarrow{n\rightarrow\infty} 0.
\end{equation*}

\end{proof}


\section*{Acknowledgments}

The authors are  very grateful to the reviewers for the very valuable remarks and comments  which lead to a significant improvement of our original manuscript.
The work has benefited from a visit of Q. Liu to the School of Mathematical Sciences, Beijing Normal University, and a visit of Z. Gao to Laboratoire de Math\'ematiques de Bretagne Atlantique,
Universit\'e  de Bretagne-Sud. The  hospitality and support of both universities
have been well appreciated.


\end{document}